\documentclass[AMA,STIX1COL]{WileyNJD-v2}
\newcommand{\bx}{{\mathbf x}}
\newcommand{\by}{{\mathbf y}}
\usepackage{subfigure}
\articletype{Article Type}%

\received{26 April 2016}
\revised{6 June 2016}
\accepted{6 June 2016}

\begin{document}

\title{Unconditionally stable exponential time differencing schemes
for the mass-conserving Allen-Cahn equation with nonlocal and local effects
}

\author[1,5]{Kun Jiang}

\author[2]{Lili Ju*}

\author[3]{Jingwei Li}

\author[4]{Xiao Li}

\authormark{K. Jiang \textsc{et al}}

\address[1]{\orgdiv{School of Mathematics and Statistics}, \orgname{Qilu University of Technology}, \orgaddress{\state{Jinan}, \country{China}}}

\address[2]{\orgdiv{Department of Mathematics}, \orgname{University of South Carolina}, \orgaddress{\state{SC}, \country{USA}}}

\address[3]{\orgdiv{Laboratory of Mathematics and Complex Systems and School of Mathematical Science}, \orgname{Beijing Normal University}, \orgaddress{\state{Beijing}, \country{China}}}

\address[4]{\orgdiv{Department of Applied Mathematics}, \orgname{The Hong Kong Polytechnic University}, \orgaddress{\state{Kowloon}, \country{Hong Kong}}}

\address[5]{\orgdiv{Faculty of Science}, \orgname{Beijing University of Technology}, \orgaddress{\state{Beijing}, \country{China}}}

\corres{*Lili Ju, Department of Mathematics, University of South Carolina, Columbia, SC 29208, USA. \email{ju@math.sc.edu}}


\abstract[Abstract]{It is well known that the classic Allen-Cahn equation satisfies the maximum bound principle (MBP), that is, the absolute value of its solution is uniformly bounded for all time by certain constant under suitable initial and boundary conditions. In this paper, we consider numerical solutions of the modified Allen-Cahn equation with a Lagrange multiplier of nonlocal and local effects, which not only shares the same MBP as the original Allen-Cahn equation but also conserves the mass exactly. We reformulate the model equation with a linear stabilizing technique, then construct first- and second-order  exponential time differencing schemes for its time integration. We  prove the unconditional MBP preservation and mass conservation of  the proposed schemes in the time discrete sense  and  derive their error estimates under some regularity assumptions. Various numerical experiments in two and three dimensions are also conducted to verify the theoretical results.}


\jnlcitation{\cname{%
\author{Jiang K},
\author{Ju L},
\author{Li J}, and
\author{Li X}}
(\cyear{2021}),
\ctitle{Unconditionally stable exponential time differencing schemes
for the mass-conserving Allen-Cahn equation with nonlocal and local effects},
\cjournal{Numer Methods Partial Differential Eq.}, \cvol{2021;00:1--6}.}

\maketitle


\section{Introduction}\label{sec1}

The classic Allen-Cahn equation takes the following form:
\begin{align}\label{ac}
 \partial_tu(\mathbf{x},t)=\varepsilon^2\Delta u(\mathbf{x},t)+f(u(\mathbf{x},t)), \qquad \mathbf{x}\in \Omega,\ t>0,
\end{align}
where $u(\mathbf{x},t)$ is the real-valued unknown function,
$\Omega\subset \mathbb{R}^d$ ($d=2,3$) is an open, connected and bounded domain
with the Lipschitz continuous boundary $\partial\Omega$,
$\varepsilon$ is an interfacial parameter and $f(u)=-F'(u)$ with $F(u)$ being certain nonlinear potential function.
The classic Allen-Cahn equation can be regarded as the $L^2$ gradient flow with respect to the energy functional
\begin{align}\label{1212}
E[u(\bx,t)]:=\int_{\Omega}\left(\frac{\varepsilon^{2}}{2}|\nabla u(\bx,t)|^{2}+ F(u(\bx,t))\right)\,{\rm d}\mathbf{x},
\end{align}
and its solution satisfies the energy dissipation law as follows:
\begin{align}\label{dissp}
\frac{d}{dt}E[u(\bx,t)]\leq 0.
\end{align}
The Allen-Cahn equation was originally introduced by Allen and Cahn \cite{AC1}
as a model for the phase separation process of a binary alloy under a fixed temperature. Since then
the Allen-Cahn equation has  been  intensively studied
due to its connection to the celebrated curvature driven geometric flow.
In the past few decades, many works on the Allen-Cahn equation  have been devoted to motions of interfaces, especially, motion by mean curvature, and numerous applications ranging from image processing \cite{mcm1, mcm2}, material sciences \cite{AC1} to biology \cite{mcm4}.

Many parabolic  types of equations often satisfy an important property, that is,
the solution must reach its maximum and/or minimum either at the initial time or on the boundary of the domain,
which is the well-studied maximum principle \cite{pde}.
The Allen-Cahn equation \eqref{ac}  satisfies a similar property, called the maximum bound principle (MBP) \cite{SAM1,mc3}:
if the initial data and/or the boundary values are pointwise bounded by a certain constant in absolute value,
then the absolute value of the solution is also bounded by the same constant everywhere and for all time. For example, when
the double-well potential $F(u)=\frac14(u^2-1)^2$ (and $f(u)=-F'(u)=u-u^3$) is used, the constant bounding the solution is 1,  i.e.,
$\|u(\cdot,t)\|\leq 1$ for all $t\geq 0$ if $\|u(\cdot,0)\|\leq 1$, where $\|\cdot\|$ denotes the supremum norm.
The MBP is weaker than the conventional maximum principle in the sense that
a problem satisfying a maximum principle must satisfy an MBP.
The equation \eqref{ac} with a uniformly elliptic linear operator ${L}$ replacing $\varepsilon^2\Delta$ and $f=0$ satisfies the maximum principle.
There have also been many studies devoted to maximum principle preserving numerical approximations of linear elliptic operators,
such as finite difference method \cite{mbp1,mbp2}, lumped-mass finite element method \cite{mbp4,mbp3},
collocation  method \cite{mbp5,mbp6}, and finite volume method\cite{mbp7}.
For the equation \eqref{ac} with a uniformly elliptic linear operator,
the nonlinear term $f(u)$ leads to the existence of time-invariant regions\cite{mc3}, in which
the MBP was proved  as a special invariant region of the Allen-Cahn equation.
Recently, a variety of works have been done on
whether such an MBP could be preserved by some time-stepping schemes for discretizing the Allen-Cahn equation.
The discrete MBPs of a finite difference semi-discrete scheme and its fully discrete approximations
with forward and backward Euler time-stepping methods
were obtained in one-dimensional space \cite{mbp10}.
Moreover, the first-order stabilized implicit-explicit schemes with finite difference spatial discretization
were proved to preserve the MBP \cite{mbp12},
which was then generalized \cite{mbp11} to the case with more general nonlinear terms.

The Cahn-Hilliard equation,
a fourth-order equation governed by the same energy functional \eqref{1212},
satisfies the so-called mass conservation while the Allen-Cahn equation fails to satisfy this property.
One can modify the Allen-Cahn equation to satisfy the  mass conservation by adding an extra Lagrange term of nonlocal constraint as \cite{eq1}
\begin{eqnarray}\label{mac2}
\partial_tu(\bx,t)=\varepsilon^2\Delta u(\bx,t)+f(u(\bx,t))-\frac{1}{|\Omega|}\int_\Omega f(u(\by,t))\,{\rm d}\by,
\qquad \mathbf{x}\in \Omega,\ t>0.
\end{eqnarray}
Integrating both sides of the equation \eqref{mac2} over $\Omega$,
we can see that the modified Allen-Cahn equation \eqref{mac2} conserves the total mass exactly:
\begin{eqnarray*}\label{mass}
\frac{d}{dt}\int_\Omega u(\mathbf{x},t)\,{\rm d}\mathbf{x}=0,
\end{eqnarray*}
or equivalently, $\int_\Omega u(\mathbf{x},t)\,{\rm d}\mathbf{x}\equiv \int_\Omega u(\mathbf{x},0)\,{\rm d}\mathbf{x}$.
In addition, the solution to the modified Allen-Cahn equation \eqref{mac2}  also satisfies the same energy dissipation laws\cite{eq1}
\eqref{dissp}
as the classic Allen-Cahn equation \eqref{ac}.
However, a drawback of such modification is that the value of the solution to \eqref{mac2} may fall beyond the interval $[-1,1]$ even for the commonly used double-well potential case \cite{LJCF20,eq1}.

Apart from
the equation \eqref{mac2}, for the Allen-Cahn equation with the double well potential, another well-known  modification is to impose a Lagrange multiplier\cite{MAC2,1} as follows:
\begin{align}\label{MAC31}
\partial_tu(\bx,t)=\varepsilon^2\Delta u(\bx,t)+f(u(\bx,t))-\frac{\int_\Omega f(u(\by,t))\,{\rm d}\by}{\int_\Omega \sqrt{4F(u(\by,t))} \,{\rm d}\by}\sqrt{4F(u(\bx,t))},\qquad \mathbf{x}\in \Omega, \ t>0,
\end{align}
where $f(u)=u-u^3$ and $F(u) =\frac14 (1-u^2)^2$.
 It is easy to show that the total mass is also exactly conserved for \eqref{MAC31}. Furthermore,
the conservation of mass is ensured by the nonlocal effect of the Lagrange multiplier $-\frac{1}{|\Omega|}\int_\Omega f(u)\,{\rm d}\by$ in \eqref{mac2},
whereas the Lagrange multiplier  in \eqref{MAC31} combines both nonlocal and local effects.
Alfaro and Alifrangis  have proven that
the solution to \eqref{MAC31} satisfies the same MBP\cite{MAC3} as that for the classic Allen-Cahn equation \eqref{ac}, that is, $\|u(\cdot,t)\|\leq 1$ for all $t\geq 0$ if $\|u(\cdot,0)\|\leq 1$. However, the  dissipation law with respect to the original energy functional \eqref{1212} does not hold theoretically for the equation \eqref{MAC31}; instead, the equation \eqref{MAC31} is the $L^2$
gradient flow with respect to a slightly different energy functional modified from \eqref{1212} \cite{1}.

There have been quite a few researches denoted to numerical schemes for the mass-conserving Allen Cahn equations \eqref{mac2}.
Kim et al. \cite{MCH1} proposed a practically unconditionally stable hybrid scheme
with an exact mass-conserving update at each time step.
Zhai et al. \cite{MCH2, MCH3} proposed the Crank-Nicolson and operator splitting schemes.
Lee \cite{MCH4} discretized the equation
by using a Fourier spectral method in space and first-, second-, and third-order implicit explicit Runge-Kutta schemes in time.

Recently, the exponential time differencing (ETD)  (or say, the exponential integrator)
has been considered for constructing unconditionally MBP-preserving schemes for the classic Allen-Cahn equation.
The ETD method comes from the variation-of-constants formula with the nonlinear terms approximated by polynomial interpolations,
followed by exact integration of the resulting integrals.
The ETD schemes have been systematically studied \cite{etd6}
and further developed by Cox and Matthews for the applications to stiff systems \cite{etd51}.
Hochbruck and Ostermann provided several nice reviews on  ETD Runge-Kutta method \cite{etd7}
and ETD multistep method \cite{etd8} for semilinear parabolic problems and the convergence of these methods were analyzed.
Du and Zhu \cite{etd9,etd3} investigated the linear stabilities of some ETD and modified ETD schemes
for the Allen-Cahn equation in two- and three-dimensional spaces.
One advantage of the ETD schemes is the exact evaluation of the linear part
so that they possess good stability and accuracy even though the linear terms have strong stiffness.
Thus,  ETD schemes have been successfully applied to phase-field models
which often yield highly stiff ODE systems under suitable spatial discretization.
Some high-order numerical methods based on fast and stable  ETD schemes were developed for solving the Allen-Cahn equation\cite{ju1}, the Cahn-Hilliard equation \cite{ju2},  the elastic bending energy model \cite{ju3}, and the no-slope-selection thin film equation \cite{wangc1,wangc2}.
A localized compact ETD method was firstly presented  \cite{ju4} for   time integration
with large step sizes for phase-field simulations of  coarsening dynamics on the Sunway TaihuLight supercomputer.
In addition, MBP-preserving numerical schemes have been also studied for the fractional Allen-Cahn equation
with the Crank-Nicolson time-stepping \cite{etd4}, the nonlocal Allen-Cahn equation by using first- and second-order ETD schemes  \cite{etd5}, and the conservative Allen-Cahn equation \eqref{mac2} using the ETD schemes \cite{LJCF20}. In a very recent work, an abstract framework was established \cite{SAM1} for analyzing the MBPs of semilinear parabolic equations and unconditionally MBP-preserving ETD schemes, and it was claimed that the classic ETD methods with order higher than $2$ cannot preserve the MBP unconditionally.  Several third- and fourth-order MBP-preserving schemes were developed for the Allen-Cahn equation by considering the integrating factor Runge-Kutta  schemes
\cite{lixiao1,li3,zhangh1}. An arbitrarily high-order ETD multistep method was presented in \cite{zhouz1} by enforcing the maximum bound via an extra cutoff postprocessing.

In this paper, we are interested in  developing stable linear  schemes for solving  the mass-conserving Allen-Cahn equation \eqref{MAC31}
based on the ETD approach. The rest of the paper is organized as follows.
In Section 2, we first reformulate the model equation \eqref{MAC31} based on the linear stabilizing technique,
and then propose first- and second-order ETD schemes for time integration of  the transformed equation, which are  shown to be unconditionally mass-conserved and MBP-preserving in the time discrete sense.  In Section 3, we prove the convergence of the proposed ETD schemes  under certain regularity assumptions. Various numerical experiments in two and three dimensions are performed  in Section 4 to validate the theoretical results.
Finally, some concluding remarks are drawn in Section 5.

\section{Unconditionally MBP-preserving exponential time differencing schemes}

Let us restate the mass-conserving Allen-Cahn equation with local and nonlocal effects as follows:
\begin{align}\label{MAC3}
\partial_tu(\bx,t)=\varepsilon^2\Delta u(\bx,t)+\bar{f}[u](\bx,t),\qquad \mathbf{x}\in \Omega, \ t>0,
\end{align}
with
\begin{align}\label{nonlocal-func}
\bar{f}[u](\bx,t)=f(u(\bx,t))-\frac{\int_\Omega f(u(\by,t))\,{\rm d}\by}{\int_\Omega g(u(\by,t))\,{\rm d}\by}g(u(\bx,t)).
\end{align}
where  $f(u)=u-u^3$ and  $g(u) = \sqrt{4F(u)} = 1-u^2$ (the notation of absolute value is dropped off since $1-u^2 \geq 0$   due to the MBP in the time-space continuous setting),
 subject to the initial value condition
\begin{align}
u(\mathbf{x},0)=u_0(\mathbf{x}),\qquad \mathbf{x}\in \overline{\Omega},
\end{align}
for some $u_0\in C(\overline\Omega)$  with $\overline{\Omega}=\Omega \cup \partial \Omega$.
We impose either the periodic boundary condition (such as a regular rectangular domain $\Omega=\prod\limits_{i=1}^d(a_i,b_i)$)
or homogeneous Neumann boundary condition given by
\begin{align*}
\frac{\partial u(\bx,t)}{\partial \mathbf{n}}=0,\quad  \mathbf{x}\in\partial\Omega, \ t\geq0,
\end{align*}
where $\mathbf{n}$ is the outer unit normal vector on the boundary $\partial\Omega$.
 Integrating both sides of the equation \eqref{MAC3} over $\Omega$, it is easy to verify
its mass-conserving property:
$$\int_{\Omega}u(\bx,t)\,{\rm d}\bx = M_0:=\int_{\Omega}u_0(\bx)\,{\rm d}\bx,\quad t\geq 0.$$
The nonlinear functions $f$ and $g$ are continuously differentiable and
\begin{equation}\label{bceq1}
f(-1)=f(1)=0, \quad g(-1)=g(1)=0.
\end{equation}
The MBP property with the bounding constant $1$ then becomes a result of the invariant set for the equation  \eqref{MAC3} \cite{MAC3}.
In addition, the two constant functions $u(\cdot,t)  \equiv 1$ or $u(\cdot,t) \equiv -1$ are clearly trivial solutions
to the equation \eqref{MAC3}. Hence we always assume that $\|u_0\|\leq 1$ and $|M_0|\ne |\Omega|$ (i.e., $u_0\nequiv \pm 1$) to avoid the these two trivial solution cases.

\begin{remark}\label{ad4}
In comparison with the original Allen-Cahn equation, the modified Allen-Cahn equation \eqref{MAC3} can preserve the mass conservation by introducing the local and nonlocal Lagrange multiplier.
However, the energy dissipation law does not hold with respect to the orginal  energy functional \eqref{1212}. To the best of our knowledge, the mass-conserving Allen-Cahn equation (6) with the double-well potential function has been proved to possess the MBP \cite{MAC3}, while whether or not the MBP also holds for the logarithmic potential case or other forms is still an open question. Therefore, in this paper we only focus on studying  the unconditional MBP-preserving schemes for the double-well potential function case.
\end{remark}

\subsection{Linear splitting for stabilization}

Let us define
\begin{equation}
\lambda_{u}(t)=\frac{\int_\Omega f(u(\bx,t))\,{\rm d}\bx}{\int_\Omega g(u(\bx,t))\,{\rm d}\bx},
\end{equation}
then we can write the equation \eqref{MAC3} as
\begin{equation}
\partial_tu(\bx,t)=\varepsilon^2\Delta u(\bx,t)+{f}(u(\bx,t))-\lambda_u(t){g}(u(\bx,t)).
\end{equation}
Next we present a result on the boundedness of $\lambda_u(t)$.

\begin{lemma}\label{lem-lambda}
For any function $\xi \in C(\overline\Omega)$ with $\|\xi\|\leq 1$ and $\xi\nequiv \pm 1$, it holds that
\begin{equation}
\label{cond_g}
|\lambda_{\xi}|\leq1.
\end{equation}
\end{lemma}

\begin{proof}
Since $\|\xi\|\leq1$ we have $g(\xi(\bx))=1-\xi^2(\bx)\geq 0$ for any $\bx\in\overline\Omega$. Furthermore,
it is clear  $\int_\Omega g(\xi(\bx))\,{\rm d}\bx  > 0$ since $\xi \in C(\overline\Omega)$ and $\xi\nequiv \pm 1$.
Thus we have
\begin{align*}
|\lambda_{\xi}|=&\left|\frac{\int_\Omega f(\xi(\bx))\,{\rm d}\bx}{\int_\Omega g(\xi(\bx))\,{\rm d}\bx}\right|=\frac{\left|\int_\Omega \xi(\bx)-\xi(\bx)^3\,{\rm d}\bx\right|}{\int_\Omega 1-\xi(\bx)^2\,{\rm d}\bx}\\[2pt]
\leq& \frac{\int_\Omega \|\xi\|(1-\xi(\bx)^2)\,{\rm d}\bx}{\int_\Omega 1-\xi(\bx)^2\,{\rm d}\bx}=  \|\xi\|\leq 1,
\end{align*}
which completes the proof.
\end{proof}

\begin{remark}
Note that for any function $\xi \in C(\overline\Omega)$ with $\|\xi\|\leq 1$,   the condition $\xi\nequiv \pm 1$ is equivalent to
$|\int_{\Omega}\xi(\bx)\,{\rm d}\bx| \ne |\Omega|$.
In the case of the constant functions $\xi\equiv \pm 1$, the above result can be understood in the limit sense.
The  boundedness of $\lambda_{u}(t)$ plays an important role on ensuring that the solutions to the mass-conserving Allen-Cahn equation \eqref{MAC3} and the corresponding temporally discretized equation analyzed later are always located in the interval $[-1,1]$.
\end{remark}

Next, let us introduce the stabilizing constant $\kappa>0$.
Correspondingly, the mass-conserving Allen-Cahn equation \eqref{MAC3} can be written in the following equivalent form
\begin{align}\label{stab-converse}
\partial_tu(\bx,t)=\mathcal{L}_\kappa u(\bx,t)+\mathcal{N}[u](\bx,t), \qquad \bx\in \Omega, \ t>0,
\end{align}
where the linear operator
\begin{align*}
\mathcal{L}_\kappa=\varepsilon^2\Delta-\kappa\mathcal{I}
\end{align*}
and the nonlinear term
\begin{align*}
\mathcal{N}[u](\bx,t)=\;&\kappa u(\bx,t)+\bar{f}[u](\bx,t)\\
=\;&\kappa u(\bx,t)+f(u(\bx,t))-\lambda_u(t)g(u(\bx,t)).
\end{align*}

We require that the stabilizing constant $\kappa$ always  satisfies
\begin{align}\label{coef}
\kappa\geq \max_{|\xi|\leq 1}(|f'(\xi)|+|g'(\xi)|)=\max_{|\xi|\leq 1}(|1-3\xi^2|+|-2\xi|)=2+2=4,
\end{align}
Then we have the following lemma on the nonlinear term.

\begin{lemma}\label{lem-nonlinear-bound}
Suppose that  the requirement \eqref{coef} holds.  For any function $\xi\in C(\overline\Omega)$ with $\|\xi\|\leq 1$ and $\xi\nequiv \pm 1$, we have
\begin{equation}
\|\mathcal{N}[\xi]\|\leq\kappa.
\end{equation}
\end{lemma}

\begin{proof}
For any $\xi\in C(\overline\Omega)$ such that $\|\xi\|\le1$,  we have from \eqref{coef} that
\begin{equation}\label{der}
0\leq\kappa +f'(\xi(\bx))-\lambda_{\xi} g'(\xi(\bx))\leq2\kappa,\quad\forall\,\bx\in\overline\Omega,
\end{equation}
where we have used the result $\|\lambda_{\xi}\|\leq1$ guaranteed by Lemma \ref{lem-lambda}.
Then, the combination of  \eqref{bceq1} and \eqref{der} yields
\begin{eqnarray}
-\kappa=-\kappa +f(-1)-\lambda_{\xi}g(-1)\hspace{-0.1cm}&&\leq \mathcal{N}[\xi](\bx)=\kappa\xi(\bx) +f(\xi(\bx))-\lambda_{\xi} g(\xi(\bx))
\nonumber\\
&&\leq\kappa +f(1)-\lambda_{\xi} g(1)=\kappa
\end{eqnarray}
for any $\bx\in\overline\Omega$, which completes
 the proof.
\end{proof}

\subsection{Exponential time differencing for time integration}

Now we propose and analyze first- and second-order linear schemes for time integration of the mass-conserving Allen-Cahn equation \eqref{MAC3} based on the equivalent form \eqref{stab-converse} and the exponential time differencing approach.

Let us divide the time interval by $\{t_n=n\tau\}_{n\geq0}$ with a time step size $\tau>0$.
The essence of the ETD method is to approximate the nonlinear operators $\mathcal{N}[u]$ by some interpolation.
We define $w(\mathbf{x},s)=u(\mathbf{x},t_n+s)$ for $s\in[0,\tau]$,
then we have the following problem:
\begin{align}\label{etd}
\left\{
\begin{array}{ll}
\partial_s w=\mathcal{L}_\kappa w+\mathcal{N}[w],&\quad \mathbf{x}\in \Omega,\,s\in (0,\tau],\\[2pt]
w(\mathbf{x},0)=u(\mathbf{x},t_n),&\quad\mathbf{x}\in \overline\Omega,\\
\end{array}
\right.
\end{align}
equipped with the periodic boundary condition or homogeneous Neumann boundary condition.

Setting $\mathcal{N}[u(t_n+s)]\approx\mathcal{N}[u(t_n)]$ in \eqref{etd} gives the first-order ETD (ETD1) scheme:
for $n\geq0$ and given $u^n$, find $u^{n+1}=w^n(\tau)$ solving
\begin{align}\label{etd1}
\left\{
\begin{array}{ll}
\partial_s w^n=\mathcal{L}_\kappa w^n+\mathcal{N}[u^n],
&\quad \mathbf{x}\in \Omega,\, s\in (0,\tau],\\[2pt]
w^n(\mathbf{x},0)=u^n,&\quad\mathbf{x}\in \overline{\Omega},\\
\end{array}
\right.
\end{align}
subject to the periodic or homogeneous Neumann boundary condition,
where $u^n$ represents an approximation of $u(t_n)$ and $u^0=u_0(\cdot)$ is given.

First we have  the following lemma regarding the Laplace operator.

\begin{lemma}\cite{SAM1}\label{lem-linear}
For any $w\in\{ u\in C(\Omega)~|~\Delta u\in C(\Omega)\}$ and $\bx_0\in \Omega$, if
\begin{eqnarray*}
w(\bx_0)=\max_{\mathbf{x}\in \overline{\Omega}}w(\mathbf{x}),
\end{eqnarray*}
then $\Delta w(\bx_0)\leq0$.
The Laplace operator $\Delta$, enforced by the periodic or homogeneous Neumann boundary condition,
generates a contraction semigroup $\{e^{\Delta t}\}_{t\geq0}$ with respect to the supremum norm on the subspace of $C(\overline{\Omega})$\cite{SAM1}, and for any $\alpha\geq0$, it holds that
\begin{align}\label{contraction-map}
\|e^{t(\Delta-\alpha)}u\|\leq e^{-\alpha t}\|u\|, \quad t\geq0,
\end{align}
for any $u\in C(\overline{\Omega})$.
\end{lemma}

\begin{proposition}[Mass conservation of the ETD1 scheme]
\label{thm-etd1-mass}
The ETD1 scheme \eqref{etd1} conserves the mass unconditionally, i.e.,
for any time step size $\tau>0$, the ETD1 solution satisfies
\begin{equation}\label{etd1-mass}
\int_\Omega u^{n}\,{\rm d}\mathbf{x}=M_0,\quad\forall\,n\geq 0.
\end{equation}
\end{proposition}

\begin{proof} By induction, assuming that $\displaystyle\int_\Omega u^n\,{\rm d}\mathbf{x}=M_0$  is given,
we only need to show $\displaystyle\int_\Omega u^{n+1}\,{\rm d}\mathbf{x}=M_0$.
Taking the $L^2$ inner product of \eqref{etd1} with $1$, we immediately obtain
\begin{align*}
\frac{d}{ds}\int_\Omega w^n\,{\rm d}\mathbf{x}+\kappa \int_\Omega w^n\,{\rm d}\mathbf{x}=\kappa \int_\Omega u^n\,{\rm d}\mathbf{x}=\kappa M_0.
\end{align*}
Let $V(s)=\int_\Omega w^n(s)\,{\rm d}\mathbf{x}$,  then we have
\begin{align*}
\frac{dV(s)}{ds}+\kappa V(s)=\kappa M_0,
\end{align*}
with $V(0)=M_0$.
Multiplying by the exponential term $e^{\kappa s}$ and integrating on the interval $[0,\tau]$, we immediately get
\begin{align*}
V(\tau)e^{\kappa \tau}-M_0= M_0(e^{\kappa \tau}-1),
\end{align*}
which implies $\displaystyle\int_\Omega u^{n+1}\,{\rm d}\mathbf{x}=V(\tau)=M_0$.
\end{proof}

Proposition \ref{thm-etd1-mass} implies  that if $u_0\nequiv \pm 1$ (i.e., $|M_0|\ne |\Omega|$), then $u^n\nequiv \pm 1$
for any $n\geq 0$.

\begin{theorem}[Discrete MBP of the ETD1 scheme]
\label{thm-etd1-mbp}
Suppose that  the requirement \eqref{coef} holds and $\|u_0\|\leq 1$ with $|M_0|\ne |\Omega|$. Then
the ETD1 scheme \eqref{etd1} preserves the discrete MBP unconditionally, i.e.,
for any time step size $\tau>0$, the ETD1 solution satisfies $\|u^n\|\leq1$ for any $n\geq0$.
\end{theorem}

\begin{proof}
By induction, we just need to show that $\|u^n\|\leq1$ and $u^n\nequiv \pm 1$ deduce $\|u^{n+1}\|\leq1$ for any $n$. The integration form of the ETD1 scheme \eqref{etd1} reads as
\begin{equation}\label{etd1-12349}
u^{n+1}=e^{\tau\mathcal{L}_\kappa}u^n+\int_{0}^\tau e^{(\tau-s)\mathcal{L}_\kappa}\mathcal{N}[u^n]\,\mathrm{d} s.
\end{equation}
According to Lemmas \ref{lem-nonlinear-bound}-\ref{lem-linear} and $\|u^n\|\leq1$, we obtain
\begin{align*}
\|u^{n+1}\| & \le \|e^{\tau\mathcal{L}_\kappa}\|\|u^n\|+\int_{0}^\tau \|e^{(\tau-s)\mathcal{L}_\kappa}\|\|\mathcal{N}[u^n]\|\,\mathrm{d} s\\
& \le e^{-\kappa\tau}+\int_{0}^\tau e^{-(\tau-s)\kappa}\kappa\,\mathrm{d} s\\
&= e^{-\kappa\tau}+\kappa\frac{1-e^{-\kappa\tau}}{\kappa} =1,
\end{align*}
which completes the proof.
\end{proof}

\begin{remark}\label{ad4221}
By approximating $e^{-\tau\mathcal{L}_\kappa} \approx\mathcal{I}-\tau\mathcal{L}_\kappa$ in the ETD1 scheme \eqref{etd1-12349}, one can obtain
\[
\frac{u^{n+1}-u^n}{\tau} = \mathcal{L}_\kappa u^{n+1} + \mathcal{N}[u^n],
\]
which is exactly the standard stabilized implicit-explicit Euler (IMEX1) scheme,
linear, and also preserves the MBP \cite{mbp12} unconditionally.
Such an observation suggests that the IMEX1 scheme is actually an approximation of the ETD1 scheme,
and the ETD1 solution is more accurate since it preserves completely the exponential behavior of the linear operator and partially
the nonlinear term\cite{ju2,ju1} while the IMEX1 scheme only uses the first-order leading term.
\end{remark}

The second-order ETD scheme of Runge-Kutta (ETDRK2) type is given by: find $u^{n+1}=w^n(\tau)$ solving
\begin{align}\label{etd2rk}
\left\{
\begin{array}{ll}
\partial_s w^n=\mathcal{L}_\kappa w^n+\Big(1-\frac{s}{\tau}\Big)\mathcal{N}[u^n]+\frac{s}{\tau}\mathcal{N}[\tilde{u}^{n+1}],
&\quad \mathbf{x}\in \Omega, \ s\in (0,\tau],\\[2pt]
w^n(x,0)=u^n,&\quad \mathbf{x}\in \overline{\Omega},\\
\end{array}
\right.
\end{align}
with $u^0=u_0(\cdot)$, subject to the periodic or homogeneous Neumann boundary condition,
where $\tilde{u}^{n+1}$ is generated by the ETD1 scheme \eqref{etd1}.
It is worth noting that both ETD1 and ETDRK2 schemes are linear.
We now prove the mass conservation and the discrete MBP for the ETDRK2 scheme.

\begin{proposition}[Mass conservation of the ETDRK2 scheme]
\label{thm-etd2rk-mass}
The ETDRK2 scheme \eqref{etd2rk} conserves the mass  unconditionally, i.e.,
for any time step size $\tau>0$, the ETDRK2 solution satisfies
\begin{equation}\label{etd2rk-mass}
\int_\Omega u^{n}\,{\rm d}\mathbf{x}=M_0,\quad\forall\,n\geq 0.
\end{equation}
\end{proposition}

 \begin{proof}
Similar to the proof for Proposition \ref{thm-etd1-mass},
by taking the $L^2$ inner product with \eqref{etd2rk} by $1$, we have
\begin{align*}
\frac{d}{ds}\int_\Omega w^n\,{\rm d}\mathbf{x}+\kappa \int_\Omega w^n\,{\rm d}\mathbf{x}=\Big(1-\frac{s}{\tau}\Big)\kappa\int_\Omega u^n\,{\rm d}\mathbf{x}+\frac{s}{\tau}\kappa\int_\Omega\tilde{u}^{n+1}\,{\rm d}\mathbf{x}=\kappa M_0,
\end{align*}
where we have used $\displaystyle\int_\Omega \tilde{u}^{n+1}\,{\rm d}\mathbf{x}=M_0$ from Proposition \ref{thm-etd1-mass}.
Thus we  obtain $\displaystyle\int_\Omega u^{n+1}\,{\rm d}\mathbf{x}=M_0$, which completes the proof.
\end{proof}

\begin{theorem}[Discrete MBP of the ETDRK2 scheme]\label{thm-etd2rk-mbp}
Suppose that  the requirement \eqref{coef} holds, $\|u_0\|\leq 1$ with $|M_0|\nequiv |\Omega|$. Then
the ETDRK2 scheme \eqref{etd2rk} preserves the discrete MBP unconditionally, i.e.,
for any time step size $\tau>0$, the ETDRK2 solution satisfies $\|u^n\|\leq1$ for any $n>0$.
\end{theorem}
\begin{proof}
By induction, let us assume that $\|u^n\|\leq1$ and $u^n\nequiv \pm 1$ for some $n$. From the ETDRK2 scheme \eqref{etd2rk}, we have
\begin{equation}\label{th235mbp}
u^{n+1}=e^{\tau\mathcal{L}_\kappa}u^n+\int_{0}^\tau e^{(\tau-s)\mathcal{L}_\kappa}\left[\Big(1-\frac{s}{\tau}\Big)\mathcal{N}[u^n]+\frac{s}{\tau}\mathcal{N}[\tilde{u}^{n+1}]\right]\,\mathrm{d} s.
\end{equation}
According to Lemmas \ref{lem-nonlinear-bound}-\ref{lem-linear},  $\|u^n\|\leq1$ and $\|\tilde{u}^{n+1}\|\leq1$ (by Theorem \ref{thm-etd1-mbp}), we obtain

\begin{align*}
\|u^{n+1}\| & \le \|e^{\tau\mathcal{L}_\kappa}\|\|u^n\|+\int_{0}^\tau \|e^{(\tau-s)\mathcal{L}_\kappa}\|\Big[\Big(1-\frac{s}{\tau}\Big)\|\mathcal{N}[u^n]\|+\frac{s}{\tau}\|\mathcal{N}[\tilde{u}^{n+1}]\|\Big]\,\mathrm{d} s\\
& \le e^{-\kappa\tau}+\int_{0}^\tau e^{-\kappa(\tau-s)}\Big[\Big(1-\frac{s}{\tau}\Big)\kappa+\frac{s}{\tau}\kappa\Big]\,\mathrm{d} s\\
& = e^{-\kappa\tau}+\kappa\frac{1-e^{-\kappa\tau}}{\kappa} = 1,
\end{align*}
which completes the proof.
\end{proof}

\begin{remark}
Different from the IMEX1 scheme, it was shown in \cite{zhangh1} that the IMEX Runge-Kutta schemes with order greater than 1 only preserves the MBP conditionally; more precisely, their MBP preservation still has the constraint on the time step size and the spatial mesh size.
\end{remark}

\begin{remark}
As claimed in \cite{SAM1}, the classic ETD Runge-Kutta approximations with order greater than 2 fail to maintain the MBP unconditionally since the higher-order interpolation polynomials contain negative coefficients, and this also happens to the mass-conserving Allen-Cahn equation \eqref{MAC3}. In addition to the Runge-Kutta type approach, multistep methods have also been widely used to design high-order schemes for gradient flow models, such as the third-order ETD multistep scheme   and the BDF3 scheme for the no-slope-selection thin film model \cite{wangc2,hao1}. However, the ETD multistep approach is based on the extrapolation for the nonlinear term. Due to the existence of negative coefficients, the extrapolation polynomials cannot be bounded by the maxima and minima of the extrapolated data, and thus the resulting ETD multistep schemes with order greater than 1 fail to unconditionally preserve the MBP  \cite{SAM1}. More recently, the integrating factor Runge-Kutta (IFRK)  method was considered for time integrartion of the classic Allen-Cahn equation in \cite{lixiao1,li3}, which successfully gives some high-order MBP-preserving schemes, thus it remains very interesting to apply them to the mass-conserving Allen-Cahn equation \eqref{MAC3}.
\end{remark}

\subsection{Fully discrete schemes}

In the following, we briefly discuss the fully-discrete ETD schemes corresponding to \eqref{etd1} and \eqref{etd2rk}, which are also unconditionally MBP preserving.
To this end, we recall the continuity of a function defined on a set $D\subset \mathbb{R}^d$ as \cite{Ru}:
\[
w \text{~is continuous at~} \mathbf{x}^*\in D\Longleftrightarrow \forall \, \mathbf{x}_k\rightarrow \mathbf{x}^* \text{~in~} D \text{~implies~} w(\mathbf{x}_k)\rightarrow w(\mathbf{x}^*).
\]
Thus, under the same theoretical framework,
the MBP property of the mass-conserving Allen-Cahn equation \eqref{MAC3} can be further extended to the case of finite-dimensional operators in space,
such as replacing $\Delta$ by its discrete approximation denoted by $\Delta_h$.
As shown in \cite{SAM1}, it is easy to verify that the central difference operator and lumped-mass finite element operator for spatial discretization of the Laplace operator $\Delta$ also satisfy Lemma \ref{lem-linear}.
In this case, $\Delta_h$ can be simply regarded as a square matrix
and generates a contraction semigroup $\{e^{\Delta_h t}\}_{t\geq0}$ on the subspace of $C(X)$
satisfying the periodic or homogeneous Neumann boundary condition,
where $X$ is the set of all spacial grid points (boundary and interior points).
The resulting space-discrete equation of \eqref{MAC3} with $\Delta$ replaced by $\Delta_h$
becomes an ordinary differential equation (ODE) system taking the same form:
\begin{equation*}\label{converse-disc}
u_t=\varepsilon^2\Delta_h u+\bar{f}[u],\quad \mathbf{x}\in X^*, t>0
\end{equation*}
with $u(\mathbf{x},0)=u_0(\mathbf{x})$ for any $\mathbf{x}\in X$,
where $X^*=X$ for the homogeneous Neumann boundary condition and $X^*=X\cap \overline{\Omega}^+$ with $\overline{\Omega}^+=\prod\limits_{i=1}^d(a_i,b_i]$ for the periodic boundary condition.

We present below the formulas of the fully-discrete ETD1 and ETDRK2 schemes,
which can be directly implemented for computation.
Let $\mathcal{L}_{\kappa,h}=\varepsilon^2\Delta_h-\kappa\mathcal{I}$ and define the $\phi$-functions as follows:
\begin{align*}
&\phi_0(a)=e^{a},\\
 &\phi_1(a)=\frac{e^{a}-1}{a},\\
  &\phi_2(a)=\frac{e^{a}-1-a}{a^2},
\end{align*}
for any $a\not=0$.
Then the  fully-discrete ETD1 scheme is given by
\begin{align*}
u^{n+1}=e^{\tau \mathcal{L}_{\kappa, h}} u^{n}+\int_{0}^{\tau} e^{(\tau-s) \mathcal{L}_{\kappa, h}}\mathcal{N}[u^{n}] \,\mathrm{d} s,
\end{align*}
or equivalently,
\begin{align}\label{etd1-fully}
u^{n+1}=\phi_0(\tau \mathcal{L}_{\kappa,h})u^n+\tau\phi_1(\tau \mathcal{L}_{\kappa,h})\mathcal{N}[u^n],
\end{align}
and the fully-discrete ETDRK2 scheme by
\begin{align*}
\left\{
\begin{array}{ll}
\widetilde{u}^{n+1}=e^{\tau  \mathcal{L}_{\kappa, h}} u^{n}+\int_{0}^{\tau} e^{(\tau-s)  \mathcal{L}_{\kappa, h}}\mathcal{N}\left[u^{n}\right] \,\mathrm{d} s,\\[2pt]
u^{n+1}=e^{\tau  \mathcal{L}_{\kappa, h}} u^{n}+\int_{0}^{\tau} e^{(\tau-s)  \mathcal{L}_{\kappa, h}}\left\{(1-\frac{s}{\tau})\mathcal{N}[u^n]+\frac{s}{\tau}\mathcal{N}[\tilde{u}^{n+1}]\right\} \,\mathrm{d} s,
\end{array}
\right.
\end{align*}
or equivalently,
\begin{align}\label{etd2rk-fully}
\left\{
\begin{array}{ll}
\widetilde{u}^{n+1}=\phi_0(\tau\mathcal{L}_{\kappa,h})u^n+\tau\phi_1(\tau\mathcal{L}_{\kappa,h})\mathcal{N}[u^n],\\[3pt]
u^{n+1}=\widetilde{u}^{n+1}+\tau\phi_2(\tau\mathcal{L}_{\kappa,h})\left(\mathcal{N}[\widetilde{u}^{n+1}]-\mathcal{N}[u^n]\right).
\end{array}
\right.
\end{align}

\section{Error estimates}

In the following, we carry out convergence analysis for the ETD schemes \eqref{etd1} and \eqref{etd2rk} in the space-continuous setting.
We first derive some useful results as follows.

\begin{lemma}\label{lem_lip}
Let $\gamma$ be any constant such that $|\Omega|> \gamma>0$. For any $\xi_1, \xi_2\in C(\overline\Omega)$ with $\|\xi_i\|\leq 1$ and  $\int_{\Omega}g(\xi_i(\bx,t))\,{\rm d}\bx\geq \gamma$ $(i=1,2)$,   we have
\begin{equation}
\|\lambda_{\xi_1}g(\xi_1)-\lambda_{\xi_2}g(\xi_2)\|\leq C_{\gamma}\|\xi_1-\xi_2\|,
\end{equation}
where $C_{\gamma}=\frac{4|\Omega|}{\gamma}+\frac{2|\Omega|^2}{\gamma^2}$.
\end{lemma}

\begin{proof}
We first have for any $\bx\in\overline\Omega$,
\begin{align*}
& \lambda_{\xi_1}g(\xi_1(\bx)) - \lambda_{\xi_2}g(\xi_2((\bx))
= \frac{\int_\Omega f(\xi_1(\by))\,{\rm d}\by}{\int_\Omega g(\xi_1(\by))\,{\rm d}\by} g(\xi_1(\bx))
- \frac{\int_\Omega f(\xi_2(\by))\,{\rm d}\by}{\int_\Omega g(\xi_2(\by))\,{\rm d}\by} g(\xi_2(\bx)) \\
& \qquad = \bigg(\int_\Omega f(\xi_1(\by))-f(\xi_2(\by))\,{\rm d}\by\bigg)
\frac{g(\xi_1(\bx))}{\int_\Omega g(\xi_1(\by))\,{\rm d}\by}
+ \int_\Omega f(\xi_2(\by))\,{\rm d}\by
\bigg(\frac{g(\xi_1(\bx))}{\int_\Omega g(\xi_1(\by))\,{\rm d}\by}
- \frac{g(\xi_2({\bx}))}{\int_\Omega g(\xi_2(\by))\,{\rm d}\by}\bigg) \\
& \qquad = \bigg(\int_\Omega f(\xi_1(\by))-f(\xi_2(\by))\,{\rm d}\by\bigg)
\frac{g(\xi_1(\bx))}{\int_\Omega g(\xi_1(\by))\,{\rm d}\by}
+ \bigg(\int_\Omega f(\xi_2(\by))\,{\rm d}\by\bigg)
\frac{g({\xi_1}(\bx))-g({\xi_2}(\bx))}{\int_\Omega g(\xi_2(\by))\,{\rm d}\by} \\
& \qquad\quad + \int_\Omega f(\xi_2(\by))\,{\rm d}\by
\bigg(g(\xi_1(\bx))
\frac{\int_\Omega g(\xi_2(\by))- g(\xi_1(\by)){\rm d}\by}
{\int_\Omega g(\xi_1(\by))\,{\rm d}\by\int_\Omega g(\xi_2(\by))\,{\rm d}\by}\bigg) \\
& \qquad =: I_1+I_2+I_3.
\end{align*}
Notice that $|g(\xi(\bx))|\leq 1$,  $|f(\xi(\bx))|\leq 1$, $|g'(\xi(\bx))|\leq 2$ and $|f'(\xi(\bx))|\leq 2$ for any $\xi\in C(\overline\Omega)$
with $\|\xi\|\leq 1$, then we get
\begin{align*}
&|I_1|\leq \frac{1}{\gamma}|g(\xi_1(\bx))|\int_\Omega \left|f(\xi_1(\by))- f(\xi_2(\by))\right|\,{\rm d}\by\leq \frac{2|\Omega|}{\gamma}\|\xi_1-\xi_2\|,\\
& |I_2|\leq \frac{|\Omega|}{\gamma}\|f(\xi_2)\|\left| g(\xi_2(\bx))- g(\xi_1(\bx))\right|\leq \frac{2|\Omega|}{\gamma}\|\xi_1-\xi_2\|,\\
&|I_3|\leq \frac{|\Omega|}{\gamma^2}\|f(\xi_2)\| |g(\xi_1(\bx))|\int_\Omega \left|g(\xi_2(\by))- g(\xi_1(\by))\right|{\rm d}\by\leq \frac{2|\Omega|^2}{\gamma^2}\|\xi_1-\xi_2\|.
\end{align*}
By combining  the above results, we obtain
\begin{equation}
\|\lambda_{\xi_1}g(\xi_1)-\lambda_{\xi_2}g(\xi_2)\|\leq \left(\frac{4|\Omega|}{\gamma}+\frac{2|\Omega|^2}{\gamma^2}\right)\|\xi_1-\xi_2\|,
\end{equation}
which completes the proof.
\end{proof}

\begin{lemma}\label{lem-nonlinear-lip}
Suppose that the requirement \eqref{coef} holds and let $\gamma$ be any constant such that $|\Omega|> \gamma>0$. For any $\xi_1, \xi_2\in C(\overline\Omega)$ with $\|\xi_i\|\leq 1$ and  $\int_{\Omega}g(\xi_i(\bx,t))\,{\rm d}\bx\geq \gamma$ $(i=1,2)$, we have
\begin{equation}
\|\mathcal{N}[\xi_1]-\mathcal{N}[\xi_2]\|\leq C^*_{\gamma} \kappa\|\xi_1-\xi_2\|,
\end{equation}
where $C^*_{\gamma}=\frac32+\frac{C_{\gamma}}{4}$.
\end{lemma}
\begin{proof}
It is easy to check that for any $\bx\in\overline\Omega$,
\begin{align*}
|\mathcal{N}[\xi_1](\bx)-\mathcal{N}[\xi_2](\bx)|&= |\kappa(\xi_1(\bx)-\xi_2
(\bx))+(f(\xi_1(\bx))-f(\xi_2(\bx))-(\lambda_{\xi_1} g(\xi_1(\bx))-\lambda_{\xi_2} g(\xi_2(\bx)))|\\
&\leq\kappa|\xi_1(\bx)-\xi_2(\bx)|+|f(\xi_1(\bx))-f(\xi_2(\bx))|+|\lambda_{\xi_1}g(\xi_1(\bx))-\lambda_{\xi_2}g(\xi_2(\bx))|\\
&\leq (\kappa+2+C_{\gamma})\|\xi_1-\xi_2\|\\
&\leq\textstyle \left( \frac32+\frac{C_{\gamma}}{4}\right)\kappa\|\xi_1-\xi_2\|,
\end{align*}
where we have used Lemma \ref{lem_lip} and the requirement  $\kappa\geq 4$. The proof is completed.
\end{proof}

Next, we study the convergence for the ETD schemes \eqref{etd1} and \eqref{etd2rk}. Let $T>0$ be a given fixed terminal time.
For any  $u\in C([0,T];C(\overline\Omega))$ with $\|u(t)\|\leq 1$ and $\int_{\Omega} u(\bx,t)\,{\rm d}\bx= M_0\ne |\Omega|$ for any $t\in[0,T]$,
there always exists a constant $\gamma_u>0$ such that $\int_{\Omega} g(u(\bx,t))\,{\rm d}\bx\geq \gamma_u$ for any $t\in[0,T]$
due to the continuity and boundedness of $u$.

\begin{theorem}[Error estimate of the ETD1 scheme]
\label{thm-etd1-conv}
Suppose that  the requirement \eqref{coef} holds and $\|u_0\|\leq 1$ with $|M_0|\ne |\Omega|$.
Assume that the exact solution $u$ to the model problem \eqref{MAC3} belongs to $C^1([0,T];C(\overline\Omega))$ and let
$\{u^n\in C(\overline\Omega)\}_{n\geq0}$ be the approximate solution generated by the ETD1 scheme \eqref{etd1}. Furthermore, we also assume that there exists a constant $\gamma_d>0$ such that $\int_{\Omega} g(u^n(\bx))\,{\rm d}\bx\geq \gamma_d$ for any $n$ with $n\tau\leq T$ and define $\gamma=\min(\gamma_u,\gamma_d)$. Then we have
\begin{align}\label{con1}
\|u(t_n)-u^n\|\leq {\textstyle\frac{C^*_{\gamma}}{C^*_\gamma-1}}Ce^{(C^*_{\gamma}-1)\kappa t_n}\tau,\quad \forall \,t_n\leq T,
\end{align}
for any $\tau>0$, where the constant $C>0$ is independent of $\tau$, $\gamma$ and $\kappa$.
\end{theorem}

\begin{proof}
Let  $e_{1}^{n}=u^{n}-u(t_{n})$.
The difference between \eqref{etd1}  and \eqref{etd} yields
\begin{align} \label{etd2rk5}
  e_{1}^{n+1}=\mathrm{e}^{\mathcal{L}_\kappa \tau}  e_{1}^{n}+\int_{0}^{\tau} \mathrm{e}^{\mathcal{L}_\kappa(\tau-s)}\left\{\mathcal{N}[u^{n}]-\mathcal{N}[u(t_{n})]+R_{1}(s)\right\} \,\mathrm{d} s,
\end{align}
where  $R_{1}(s)$  is the truncation error as
$$R_{1}(s)=\mathcal{N}[u(t_{n})]-\mathcal{N}[u(t_{n}+s)], \quad s \in[0, \tau].$$
By the MBP property of $u$ and  Lemma \ref{lem-nonlinear-lip}, we have
$$ \|R_{1}(s)\|=\|\mathcal{N}[u(t_{n})]-\mathcal{N}[u(t_{n}+s)]\|
\leq C^*_{\gamma}{\kappa} \|u(t_{n})-u(t_{n}+s)\| \leq C_{1} C^*_{\gamma} \kappa\tau, \quad \forall s \in[0, \tau],$$
where   the constant  $C_{1}$  depends on the $C^{1}([0, T]; C(\overline\Omega)$ norm of $u$,  but independent of  $\tau$  and  $\kappa$.  Similarly, since $\left\|u^{n}\right\| \leq 1$ due to Theorem  \ref{thm-etd1-mbp},  we also obtain by Lemma \ref{lem-nonlinear-lip} that
\begin{align}\label{con33}
  \|\mathcal{N}[u^{n}]-\mathcal{N}[u(t_{n})]\| \leq C^*_{\gamma}\kappa\|u^{n}-u(t_{n})\|=C^*_{\gamma}\kappa \left\|e_{1}^{n}\right\|.
\end{align}
Then, we derive from \eqref{etd2rk5} and Lemma \ref{lem-linear} that
\begin{align}\label{etd1con}\nonumber
\|e_{1}^{n+1}\| \leq \,& \mathrm{e}^{-\kappa \tau}\|e_{1}^{n}\|+\int_{0}^{\tau} \mathrm{e}^{-\kappa(\tau-s)}\left\{\|\mathcal{N}[u^{n}]-\mathcal{N}[u(t_{n})]\|+\|R_{1}(s)\|\right\} \,\mathrm{d} s \\ \nonumber
 \leq \,& \mathrm{e}^{-\kappa \tau}\|e_{1}^{n}\|+C^*_{\gamma}\kappa\left(\|e_{1}^{n}\|+C_{1} \tau\right) \int_{0}^{\tau} \mathrm{e}^{-\kappa(\tau-s)} \,\mathrm{d} s \\ \nonumber
 =\,& \mathrm{e}^{-\kappa \tau}\|e_{1}^{n}\|+\frac{1-\mathrm{e}^{-\kappa \tau}}{\kappa}C^*_{\gamma}\kappa\left(\|e_{1}^{n}\|+C_{1} \tau\right) \\ \nonumber
 =\,&(C^*_{\gamma}-(C^*_{\gamma}-1)\mathrm{e}^{-\kappa \tau})\|e_{1}^{n}\|+\frac{1-\mathrm{e}^{-\kappa \tau}}{\kappa \tau}  C^*_{\gamma}C_{1} \kappa \tau^{2} \\
  \leq\; &(1+(C^*_{\gamma}-1)\kappa \tau)\|e_{1}^{n}\|+C^*_{\gamma}C_{1} \kappa \tau^{2},
\end{align}
where in the last step we have used the fact that  $1-\mathrm{e}^{-a} \leq a$ for any  $a>0$. By  induction, we have
\begin{align*}
\|e_{1}^{n}\| & \leq(1+(C^*_{\gamma}-1)\kappa \tau)^{n}\|e_{1}^{0}\|+C^*_{\gamma}C_{1} \kappa \tau^{2} \sum_{k=0}^{n-1}(1+(C^*_{\gamma}-1)\kappa \tau)^{k} \\
  &=(1+(C^*_{\gamma}-1)\kappa \tau)^{n}\|e_{1}^{0}\|+\textstyle\frac{C^*_{\gamma}C_{1}}{C^*_{\gamma}-1} \tau[(1+(C^*_{\gamma}-1)\kappa \tau)^{n}-1] \\
  & \leq \mathrm{e}^{(C^*_{\gamma}-1)\kappa n \tau}\|e_{1}^{0}\|+\textstyle\frac{C^*_{\gamma}}{C^*_\gamma-1}C_{1} \mathrm{e}^{(C^*_{\gamma}-1)\kappa n \tau} \tau.
\end{align*}
Finally we obtain (\ref{con1}) by letting $C=C_1$ since  $e_{1}^{0}=0$  and  $n \tau=t_{n}$.
\end{proof}

\begin{theorem}[Error estimate of the ETDRK2 scheme]
\label{thm-etd2rk-conv}
Suppose that  the requirement \eqref{coef} holds and $\|u_0\|\leq 1$ with $|M_0|\ne |\Omega|$.
Assume that the exact solution $u$ to the model problem \eqref{MAC3} belongs to $C^2([0,T];C(\overline\Omega))$ and let
$\{u^n\in C(\overline\Omega)\}_{n\geq0}$ be the approximate solution generated by the ETD2 scheme \eqref{etd2rk}. Furthermore, we also assume that there exists a constant $\gamma_d>0$ such that $\int_{\Omega} g(u^n(\bx))\,{\rm d}\bx\geq \gamma_d$ for any $n$ with $n\tau\leq T$ and define $\gamma=\min(\gamma_u,\gamma_d)$. Then we have
\begin{align}\label{etd2conver}
\|u(t_n)-u^n\|\leq Ce^{(C^*_\gamma-1)\kappa t_n}\tau^2, \quad \forall \,t_n\leq T,
\end{align}
for any $\tau>0$, where the constant $C>0$ is independent of $\tau$.
\end{theorem}

\begin{proof}
The proof strategy is quite similar to that for   the  ETD1 scheme. Let
$e_{2}^{n}=u^{n}-u\left(t_{n}\right)$,  then we have
\begin{align} \label{con22}
 e_{2}^{n+1}=\;& \mathrm{e}^{\mathcal{L}_\kappa \tau} e_{2}^{n}+\int_{0}^{\tau} \mathrm{e}^{\mathcal{L}_\kappa(\tau-s)} \left\{\left(1-\frac{s}{\tau}\right)\left(\mathcal{N}[u^{n}]-\mathcal{N}[u(t_{n})]\right)+\frac{s}{\tau}\left(\mathcal{N}[\tilde{u}^{n+1}]-\mathcal{N}[u(t_{n+1})]\right)+R_{2}(s)\right\} \,\mathrm{d} s,
\end{align}
where $R_{2}(s)$ is the truncation error given by
$$R_{2}(s)=\left(1-\frac{s}{\tau}\right) \mathcal{N}[u(t_{n})]+\frac{s}{\tau} \mathcal{N}[u(t_{n+1})]-\mathcal{N}[u(t_{n}+s)], \quad s \in[0, \tau]. $$
Using the estimation of the linear interpolation, we have
$$\|R_{2}(s)\|\leq C_{2}  \tau^{2}, \quad \forall \,s \in[0, \tau],$$
where the constant  $C_{2}$ depends on the  $C^{2}([0, T], C(\overline\Omega))$ norm of  $u$ and $\gamma$ but is independent of $\tau$.
From the last inequality in \eqref{etd1con}, we know
$$\|\tilde{u}^{n+1}-u(t_{n+1})\| \leq
(1+(C^*_{\gamma}-1)\kappa \tau)\|u^{n}-u(t_{n})\|+C^*_{\gamma}C_{1} \kappa \tau^{2}.$$
By combining the above inequality with  Theorem \ref{thm-etd2rk-mbp}  and Lemma \ref{lem-nonlinear-lip},  we have, for any  $s \in[0, \tau],$
\begin{align*}
&\left\|\left(1-\frac{s}{\tau}\right)\left(\mathcal{N}[u^{n}]-\mathcal{N}[u(t_{n})]\right)+\frac{s}{\tau}\left(\mathcal{N}[\tilde{u}^{n+1}]-\mathcal{N}[u(t_{n+1})]\right)\right\| \\
&\qquad \leq C^*_\gamma\kappa\left(\left(1-\frac{s}{\tau}\right) \left\|e_{2}^{n}\right\|+\frac{s}{\tau}\left((1+(C^*_\gamma-1)\kappa \tau)\left\|e_{2}^{n}\right\|+C^*_\gamma C_{1} \kappa \tau^{2}\right)\right) \\
& \qquad =C^*_\gamma\kappa\|e_2^n\|+C^*_\gamma(C^*_\gamma-1)\kappa^2 s\|e_2^n\|+{C^*_\gamma}^2C_1\kappa^2\tau s.
\end{align*}
Then, we obtain from \eqref{con22} and Lemma \ref{lem-linear} that
\begin{align*}
\begin{aligned}
\|e_{2}^{n+1}\| \leq\;& \mathrm{e}^{-\kappa \tau}\|e_{2}^{n}\|+\int_{0}^{\tau} \mathrm{e}^{-\kappa(\tau-s)}\left\{C^*_\gamma\kappa\|e_2^n\|+C^*_\gamma(C^*_\gamma-1)\kappa^2 s\|e_2^n\|+{C^*_\gamma}^2C_1\kappa^2\tau s+C_{2} \tau^{2}\right\} \,\mathrm{d} s\\
  =\,& \mathrm{e}^{-\kappa \tau}\|e_{2}^{n}\|+\left(C^*_\gamma \kappa\left\|e_{2}^{n}\right\|+C_{2}\tau^{2}\right) \int_{0}^{\tau} \mathrm{e}^{-\kappa(\tau-s)} \,\mathrm{d} s+\left(C^*_\gamma(C^*_\gamma-1) \kappa^{2}\|e_{2}^{n}\|+{C^*_\gamma}^2 C_{1} \kappa^{2} \tau\right) \int_{0}^{\tau} s \mathrm{e}^{-\kappa(\tau-s)} \,\mathrm{d} s \\
  =\,& \mathrm{e}^{-\kappa \tau}\left\|e_{2}^{n}\right\|+\frac{1-\mathrm{e}^{-\kappa \tau}}{\kappa}\left(C^*_\gamma \kappa\|e_{2}^{n}\|+C_{2}\tau^{2}\right)+\frac{\mathrm{e}^{-\kappa \tau}-1+\kappa \tau}{\kappa^{2}}\left(C^*_\gamma(C^*_\gamma-1) \kappa^{2}\|e_{2}^{n}\|+{C^*_\gamma}^2 C_{1} \kappa^{2} \tau\right) \\
  =\,&\left((C^*_\gamma-1)^2\mathrm{e}^{-\kappa \tau}+C^*_\gamma(C^*_\gamma-1) \kappa \tau-C^*_\gamma(C^*_\gamma-2)\right)\|e_{2}^{n}\|+\frac{1-\mathrm{e}^{-\kappa \tau}}{\kappa}C_{2} \tau^{2}+\frac{\mathrm{e}^{-\kappa \tau}-1+\kappa \tau}{\kappa^{2}} \cdot {C^*_\gamma}^2C_1 \kappa^{2} \tau \\
  \leq \,&\left(1+(C^*_\gamma-1)\kappa \tau+\frac12{(C^*_\gamma-1)^2}(\kappa \tau)^{2}\right)\|e_{2}^{n}\|+\left(C_{2}+\frac12{C^*_\gamma}C_{1} \kappa^{2}\right) \tau^{3},
  \end{aligned}
\end{align*}
where we have used the inequality $1-a \leq \mathrm{e}^{-a} \leq 1-a+\frac{a^{2}}{2}$  for any  $a>0$.
By induction, we obtain
\begin{align*}
\|e_{2}^{n}\| & \leq\left(1+(C^*_\gamma-1)\kappa \tau+\frac12{(C^*_\gamma-1)^2}(\kappa \tau)^{2}\right)^{n}\|e_{2}^{0}\|\\
&\quad+\left(\frac12{C^*_\gamma}C_{1} \kappa^{2}+C_{2} \right) \tau^{3} \sum_{k=0}^{n-1}\left(1+(C^*_\gamma-1)\kappa \tau+\frac12{(C^*_\gamma-1)^2}(\kappa \tau)^{2}\right)^{k} \\
  & \leq\left(1+(C^*_\gamma-1)\kappa \tau+\frac12{(C^*_\gamma-1)^2}(\kappa \tau)^{2}\right)^{n}\|e_{2}^{0}\|\\
  &\quad+\left(\frac{C^*_\gamma}{2(C^*_\gamma-1)}C_{1} \kappa+\frac{C_{2}}{(C^*_\gamma-1)\kappa}\right) \tau^{2}\left(\left(1+(C^*_\gamma-1)\kappa \tau+\frac12{(C^*_\gamma-1)^2}(\kappa \tau)^{2}\right)^{n}-1\right) \\
  &\leq  \mathrm{e}^{(C^*_\gamma-1)\kappa n \tau}\|e_{2}^{0}\|+\left(\frac{C^*_\gamma}{2(C^*_\gamma-1)}C_{1} \kappa+\frac{C_{2}}{(C^*_\gamma-1)\kappa}\right)\mathrm{e}^{(C^*_\gamma-1)\kappa n \tau} \tau^{2}.
\end{align*}
By letting  $C=\frac{C^*_\gamma}{2(C^*_\gamma-1)}C_{1} \kappa+\frac{C_{2}}{(C^*_\gamma-1)\kappa}$,  we finally obtain \eqref{etd2conver} since
$e_{2}^{0}=0$  and  $n \tau=t_{n}$.
\end{proof}

\begin{remark}\label{addcond}
In Theorems \ref{thm-etd1-conv} and  \ref{thm-etd2rk-conv}, we additionally assume that there exists a constant $\gamma_d>0$ such that $\int_{\Omega} g(u^n(\bx))\,{\rm d}\bx\geq \gamma_d$ for any $n$ with $n\tau\leq T$. While this assumption on the approximate solution $\{u^n\}$ is necessary to our current proofs of the error estimates,  it remains an  interesting question whether   such assumption can be removed with other  analysis techniques.
Here, we only give the temporal convergence analysis for the ETD1 and ETDRK2 schemes in the space-continuous setting. In the similar spirit of the analysis in \cite{etd5}, the convergence analysis for the fully discrete version is also available by taking the truncation error for spatial discretization into account.
\end{remark}

\section{Numerical experiments}\label{sec4}

In this section, we present some numerical experiments
to demonstrate the performance (convergence rates, mass conservation and MBP preservation) of the proposed ETD schemes \eqref{etd1} and \eqref{etd2rk} for the mass-conserving Allen-Cahn equation \eqref{MAC3}.
We take the domain $\Omega=(-0.5,0.5)^d$ with $d=2$ or $3$.
Moreover, the ETDRK2 scheme is used in all examples
while the ETD1 scheme is only considered in temporal convergence tests due to its lack of high accuracy.
For simplicity,  we here only consider the case of periodic boundary condition and the case of homogeneous Neumann boundary condition is similar.
The stabilizing coefficient is set to be $\kappa=4$ in all experiments.
The spatial discretization is performed by the central difference discretization to form the fully-discrete schemes \eqref{etd1-fully} and \eqref{etd2rk-fully}, in which the products of matrix exponentials with vectors are computed  using the fast Fourier transform based implementation\cite{ju1}.

\subsection{Convergence tests}

We run the first-  and second-order ETD schemes for the mass-conserving Allen-Cahn equation \eqref{MAC3} in 2D
with $\varepsilon=0.01$ and the initial value $u_0(x,y)=\cos(2\pi x)\cos(2\pi y)$. The terminal time is chosen to  be $T=1$.
In order to accurately catch the convergence rate in time,
the spatial mesh size must be small enough and we set   $h = 1/1024$.
To compute the solution errors under different time step sizes $\tau = 1/2^k$ for $k=2,3,\dots,8$, we treat the approximate solution obtained by the ETDRK2 scheme with $\tau=1/1024$ as the benchmark.
\tablename~\ref{tab-temporal-conv} reports the $L^\infty$ and $L^2$ norms of the solution errors at the terminal time $T=1$
and corresponding temporal convergence rates, which clearly verifies
the first-order temporal accuracy for ETD1  and the second-order temporal accuracy for ETDRK2  respectively.

\begin{table*}[!ht]%
\caption{$L^2$ and $L^\infty$ solution errors at $T=1$ and corresponding convergence rates in time by the ETD1 and ETDRK2 schemes respectively.}\label{tab-temporal-conv}
\centering
\begin{tabular*}{500pt}{@{\extracolsep\fill}lcccccccc@{\extracolsep\fill}}
\toprule
\multirow{2}{*}{\textbf{$\tau$}} &\multicolumn{4}{@{}c@{}}{\textbf{ETD1}} & \multicolumn{4}{@{}c@{}}{\textbf{ETDRK2}} \\\cmidrule{2-5}\cmidrule{6-9}
& \textbf{$L^2$ Error}  & \textbf{Rate} & \textbf{$L^\infty$ Error} & \textbf{Rate}  & \multicolumn{1}{@{}l@{}}{\textbf{$L^2$ Error}}  & \textbf{Rate}  & \textbf{$L^\infty$ Error} & \textbf{Rate} \\
\midrule
$1/4$     &2.28e-1&       &1.51e-1&       &1.49e-1&       &9.57e-2& \\
$1/8$  &1.57e-1&0.54&1.01e-1&0.58&6.98e-2&1.09&4.36e-2&1.13\\
$1/16$  &9.40e-2&0.74&5.95e-2&0.77&2.53e-2&1.47&1.55e-2&1.49\\
$1/32$  &5.12e-2&0.87&3.19e-2&0.90&7.72e-3&1.71&4.69e-3&1.72\\
$1/64$&2.61e-2&0.97&1.61e-2&0.98&2.13e-3&1.85&1.29e-3&1.86\\
$1/128$&1.25e-2&1.06&7.70e-3&1.06&5.57e-4&1.94&3.37e-4&1.94\\
$1/256$&5.43e-3&1.20&3.33e-3&1.20&1.36e-4&2.03&8.23e-5&2.03\\
\bottomrule
\end{tabular*}
\end{table*}

Next, we  test the spatial convergence of the central difference using the ETDRK2 scheme.
We  fix the time step size $\tau=T/1024$ and regard the approximate solution produced by the ETDRK2 scheme with $h =1/2048$ as the benchmark
for computing the solution  errors  with different spatial mesh sizes.
The $L^\infty$ and $L^2$ norms of the solution errors at $T=1$ and corresponding convergence rates are presented in \tablename~\ref{tab-spacial-conv}. It is observed that the convergence rates with respect to $h$ are clearly of second order as expected.

\begin{table}[!htb]%
\centering
\caption{$L^2$  and  $L^\infty$  solution errors at $T=1$ and corresponding convergence rates in space by the ETDRK2 scheme.}
\label{tab-spacial-conv}%
\begin{tabular*}{300pt}{@{\extracolsep\fill}lcccc@{\extracolsep\fill}}%
\toprule
\textbf{$h$} & \textbf{$L^2$ Error} & \textbf{Rate} & \textbf{$L^\infty$ Error} & \textbf{Rate} \\
\midrule
1/64    &9.98e-4 &       &3.14e-4& \\
1/128  &3.09e-4 &1.69&9.21e-5&1.77\\
1/256  &8.38e-5 &1.88&2.40e-5&1.94\\
1/512  &2.12e-5 &1.97&6.07e-6&1.98\\
1/1024&5.33e-6 &1.99&1.52e-6&1.99\\
\bottomrule
\end{tabular*}
\end{table}

\subsection{Tests of mass-conservation and MBP-preservation}

We numerically simulate and investigate the discrete MBP in long-time phase separation processes
governed by the mass-conserving Allen-Cahn equation \eqref{MAC3} in 2D and 3D spaces. The ETDRK2 scheme is used.
We set $\varepsilon=0.01$ and the time step size $\tau=0.1$.  The spatial grid size is selected to be $h = 1/1024$ in 2D and $h = 1/256$ in 3D .

We start the 2D simulations with an initial configuration of  $u_0 = 0.9\,\rm{rand\,(\cdot)}$ (here $\rm{rand\,(\cdot)}$ represents the quasi-uniform random distribution between $-1$ and $1$). In this case, we also compare the simulation results with those produced by the classic Cahn-Hilliard equation \cite{ju2}
\begin{equation}\label{ch}
\partial_t u(\bx, t)=-\Delta({\varepsilon}^2\Delta u(\bx, t)+f(u(\bx, t))), \qquad \mathbf{x}\in \Omega,\ t>0,
\end{equation}
with  ${\varepsilon}=0.01$ and the same initial configuration.
\figurename~\ref{2D1} presents the configurations of the simulated solutions at $t=1$, $100$, $1000$, and $2500$
for the mass-conserving Allen-Cahn equation \eqref{MAC3}. The  steady state is gradually reached after about $t=2000$.
The evolutions of the mass, the supremum norm and the energy are plotted in \figurename~\ref{2D2}.
It is easy to see that the mass is exactly conserved and  the discrete MBP  is  preserved perfectly along the time.
Although there is no energy dissipation law for the equation \eqref{MAC3} theoretically,
we still observe that the energy decays monotonically for this example.
The configurations of the simulated solutions at  $t=1$, $10$, $50$, and $300$ for the Cahn-Hilliard equation are presented in \figurename~\ref{2D1111},
where the same steady state is reached after  around $t=80$.
This implies that the evolution of the phase structure in the mass-conserving Allen-Cahn equation
is much slower than that in the Cahn-Hilliard equation.
\figurename~\ref{CHDR2} shows the corresponding evolutions of the mass, the supremum norm and the energy.
We observe that the mass is conserved and the energy decays monotonically along the time.
However, the supremum norm of the numerical solution is beyond the constant $1$ after about $t=1$
since the Cahn-Hilliard equation does not have the MBP property.

\begin{figure*}[!ht]
\centerline{
\subfigure[$t=1$]{{\includegraphics[width=0.33\textwidth]{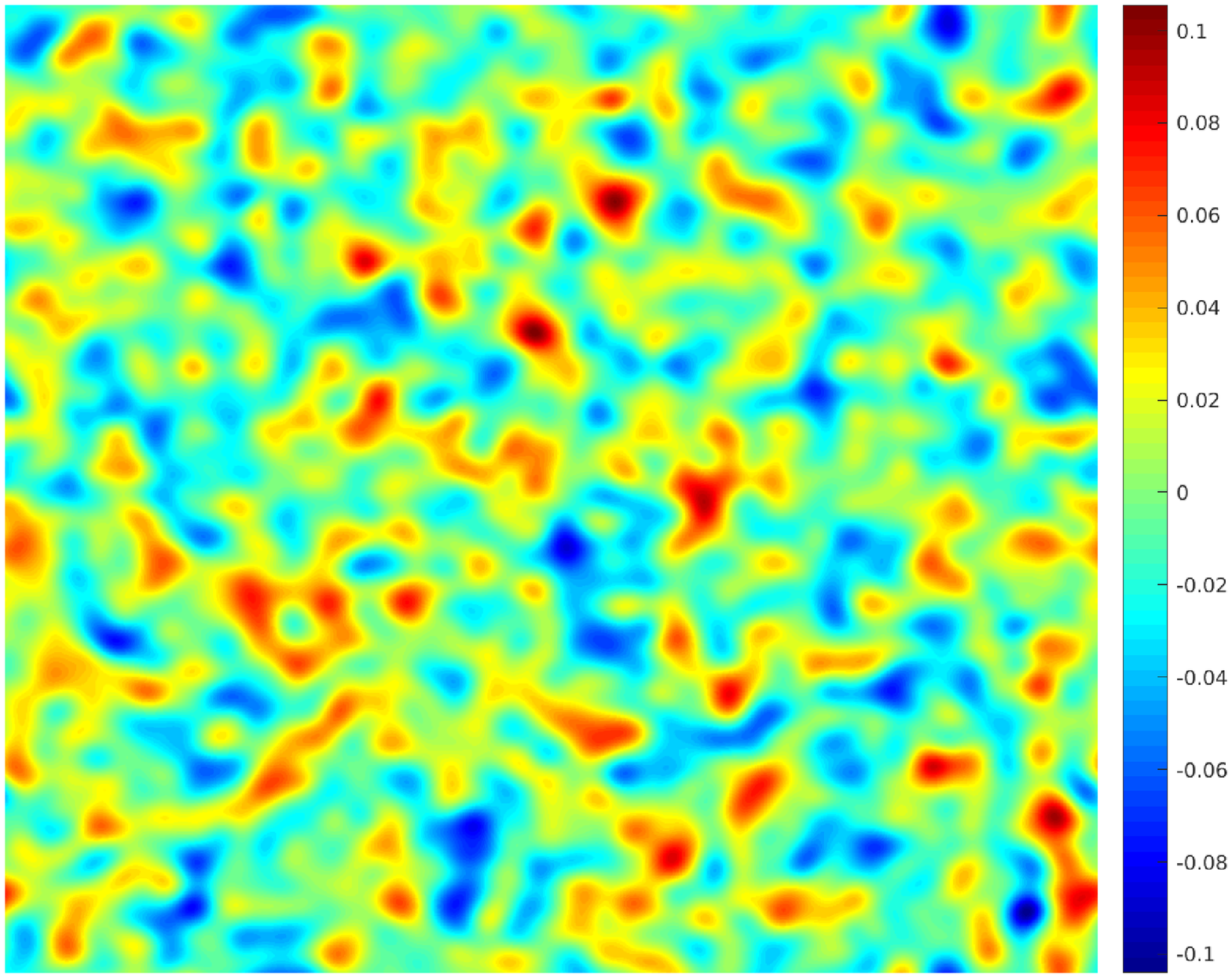}}}\hspace{0.3cm}
\subfigure[$t=100$]{{\includegraphics[width=0.33\textwidth]{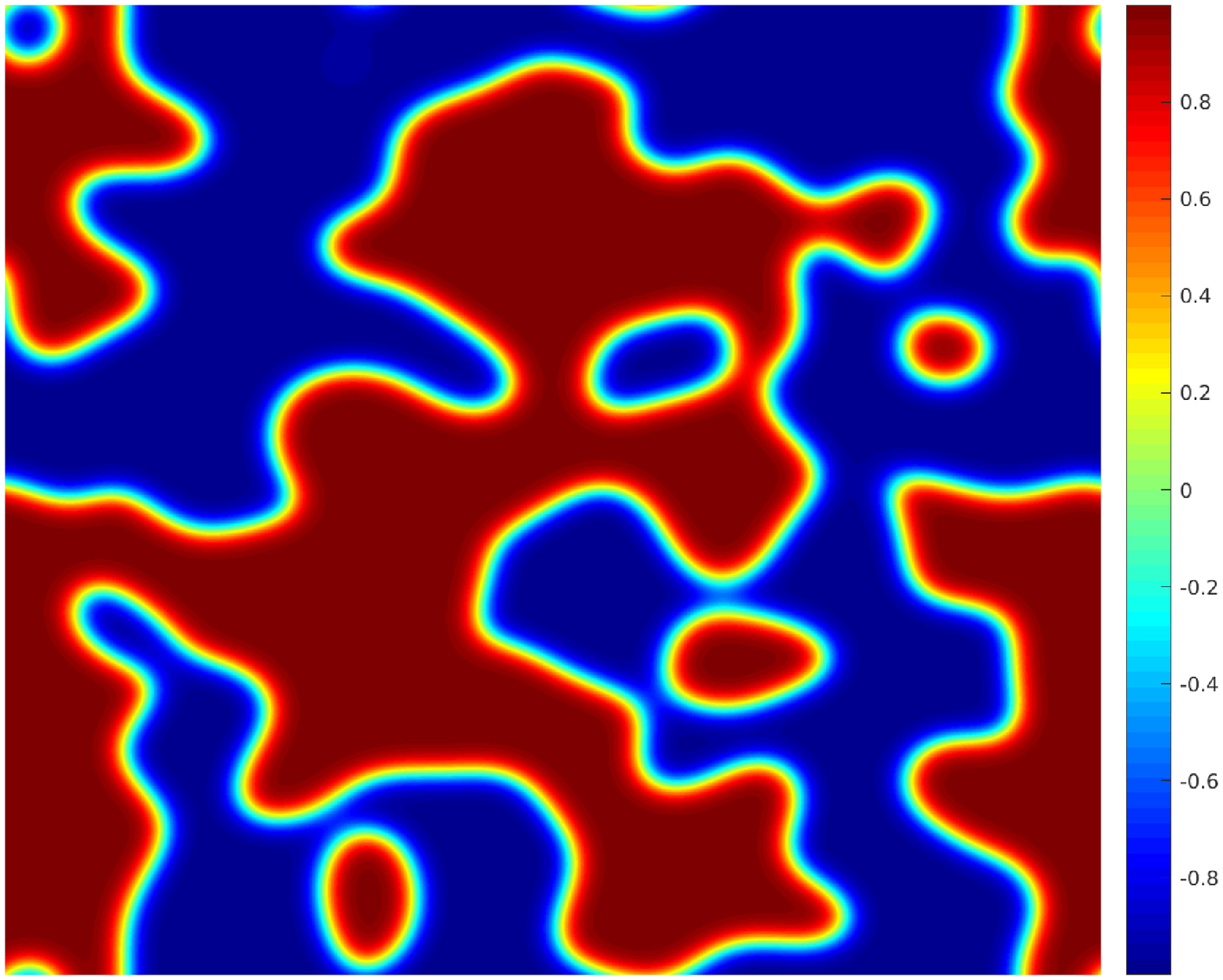}}}}
\centerline{
\subfigure[$t=1000$]{{\includegraphics[width=0.33\textwidth]{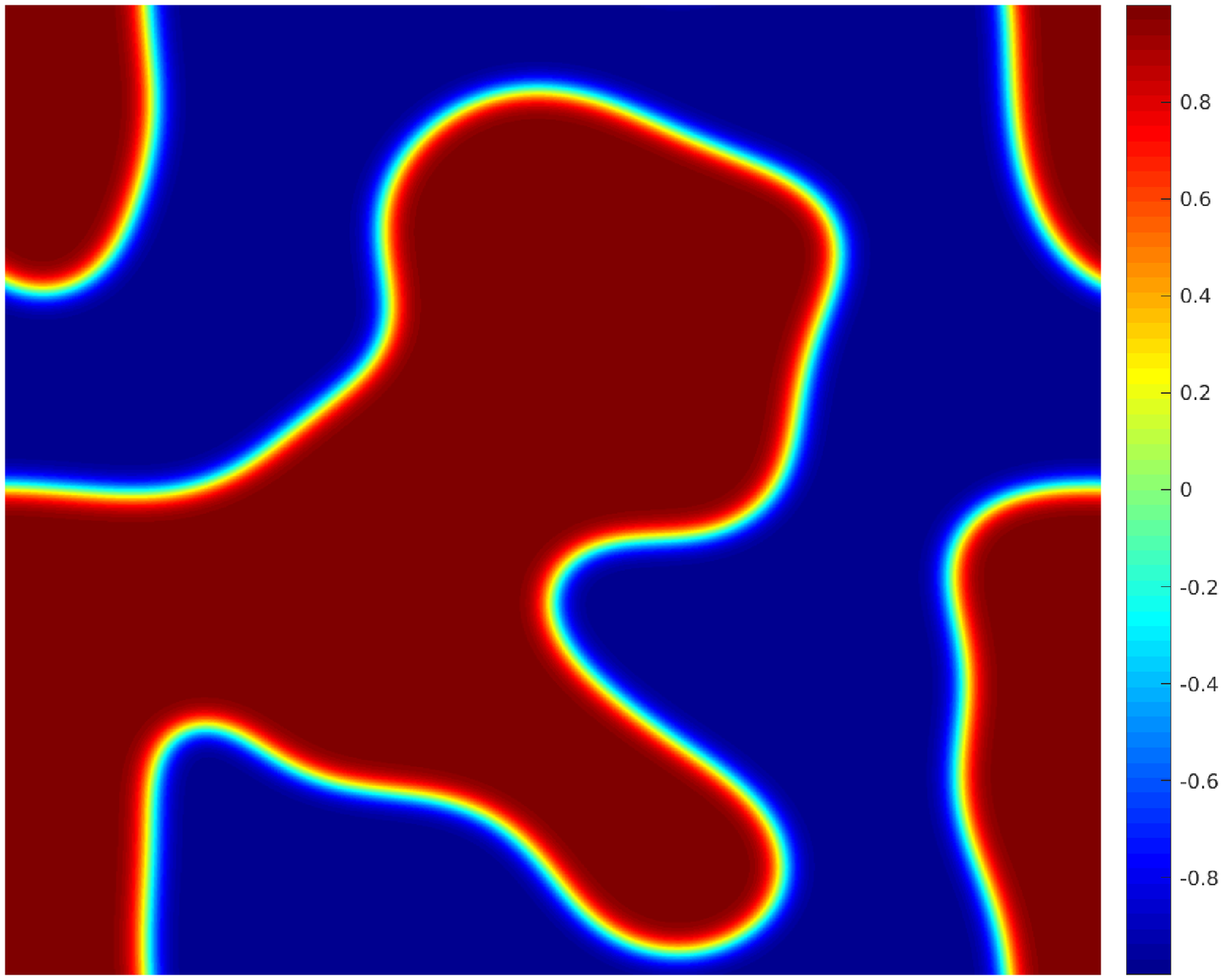}}}\hspace{0.3cm}
\subfigure[$t=2500$]{{\includegraphics[width=0.33\textwidth]{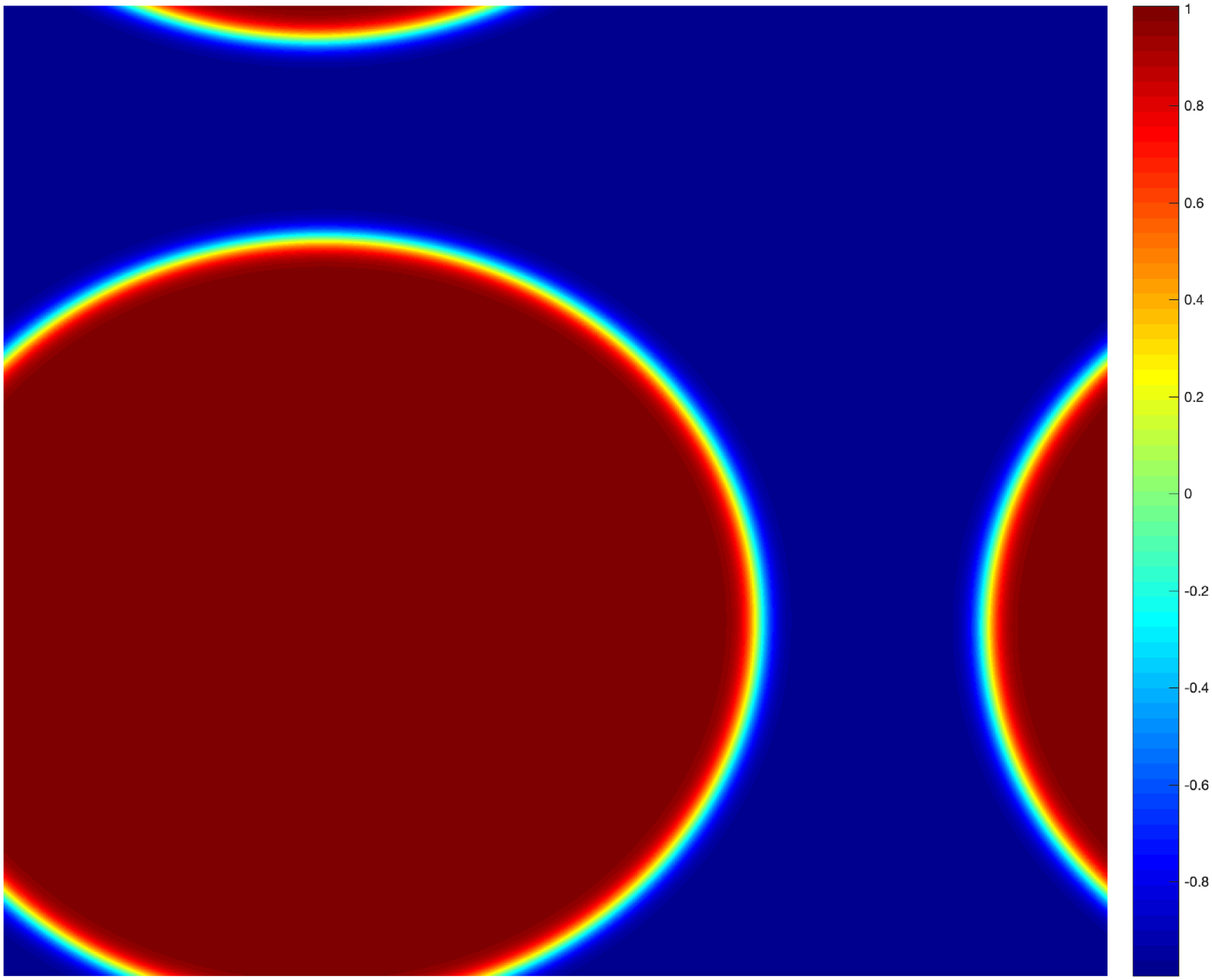}}}}
\caption{The simulated solutions at $t=1$, $100$, $1000$ and $2500$  respectively
for the mass-conserving Allen-Cahn equation with an initial quasi-uniform state  in 2D by the ETDRK2 scheme.}
\label{2D1}
\end{figure*}

\begin{figure*}[!ht]
\centerline{
\subfigure{{\includegraphics[width=0.36\textwidth]{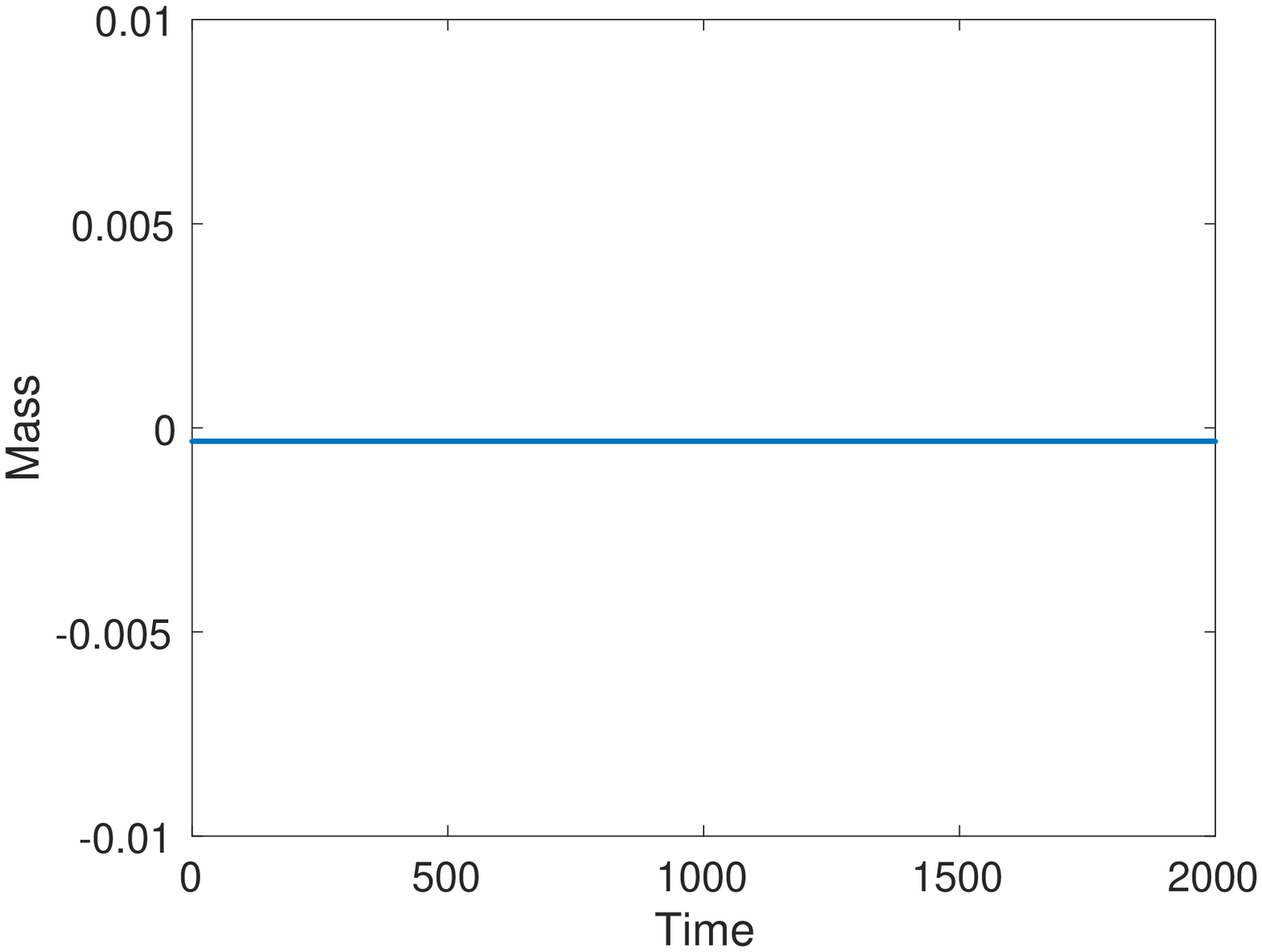}}}\hspace{-0.55cm}
\subfigure{{\includegraphics[width=0.36\textwidth]{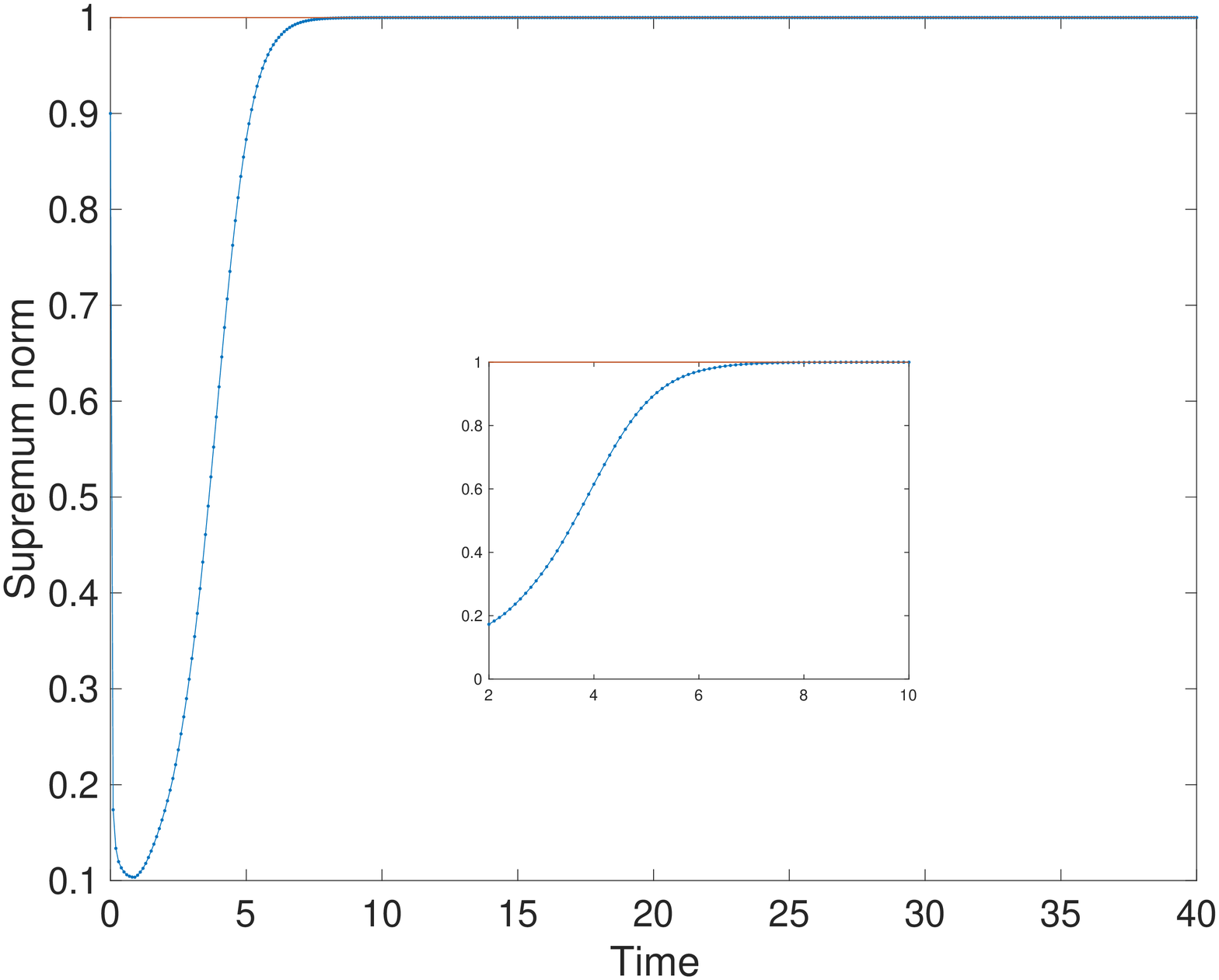}}}\hspace{-0.55cm}
\subfigure{{\includegraphics[width=0.36\textwidth]{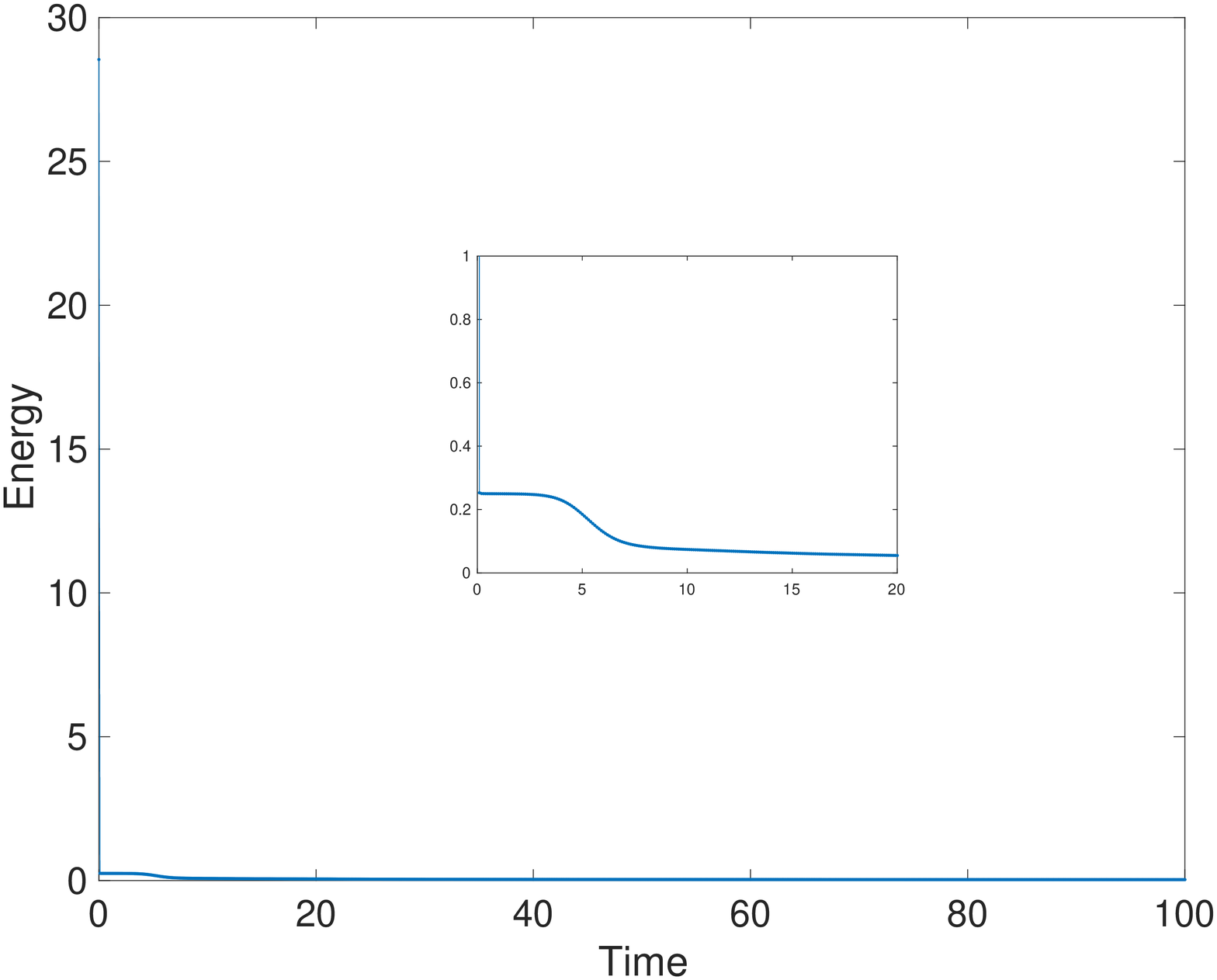}}}}
\caption{Evolutions of the mass, the supremum norm and  the  energy
of the simulated solutions for the mass-conserving Allen-Cahn equation with  an initial quasi-uniform state in 2D by the ETDRK2 scheme. }
\label{2D2}
\end{figure*}

\begin{figure*}[!ht]
\centerline{
\subfigure[$t=1$]{{\includegraphics[width=0.33\textwidth]{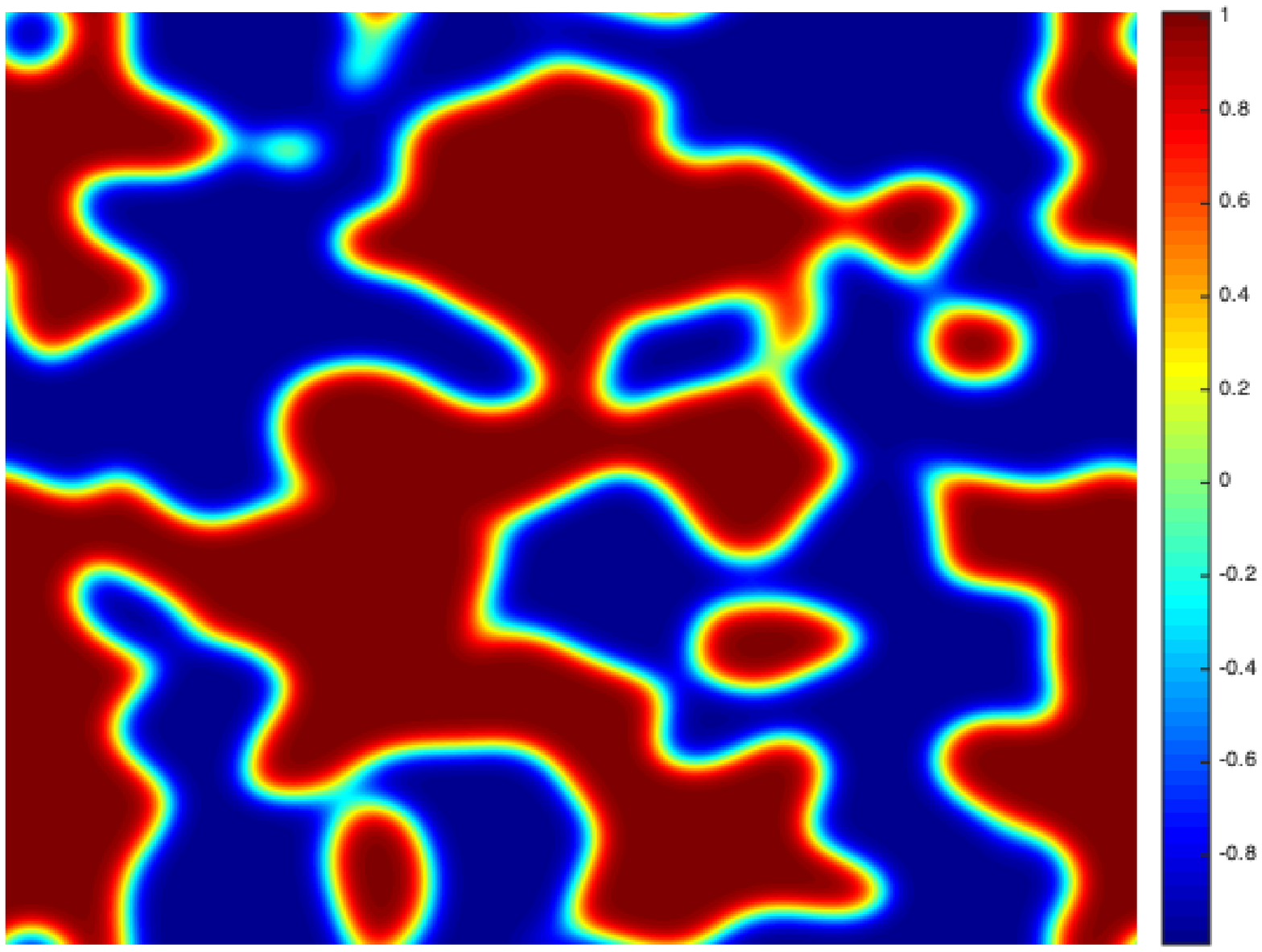}}}\hspace{0.3cm}
\subfigure[$t=10$]{{\includegraphics[width=0.33\textwidth]{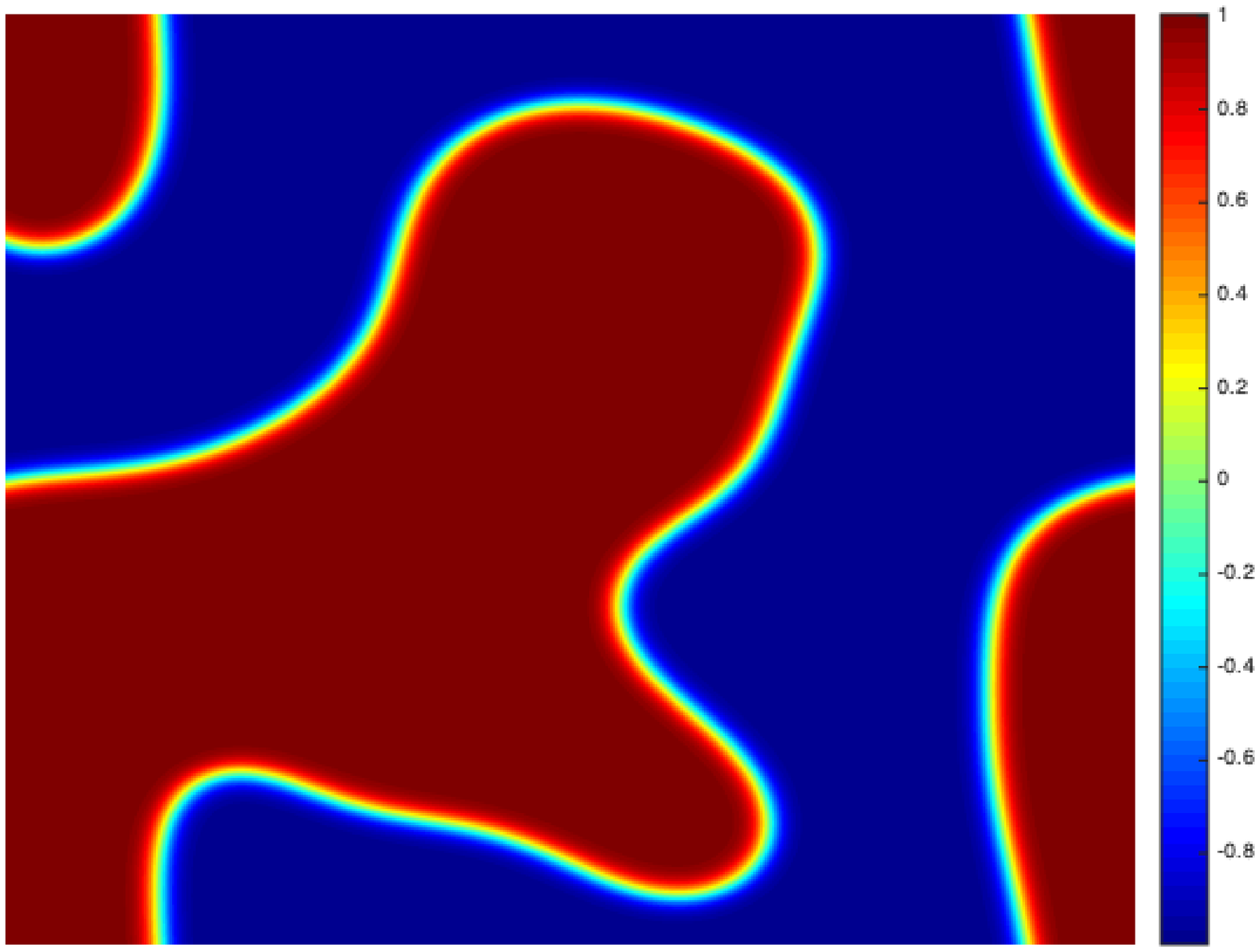}}}}
\centerline{
\subfigure[$t=50$]{{\includegraphics[width=0.33\textwidth]{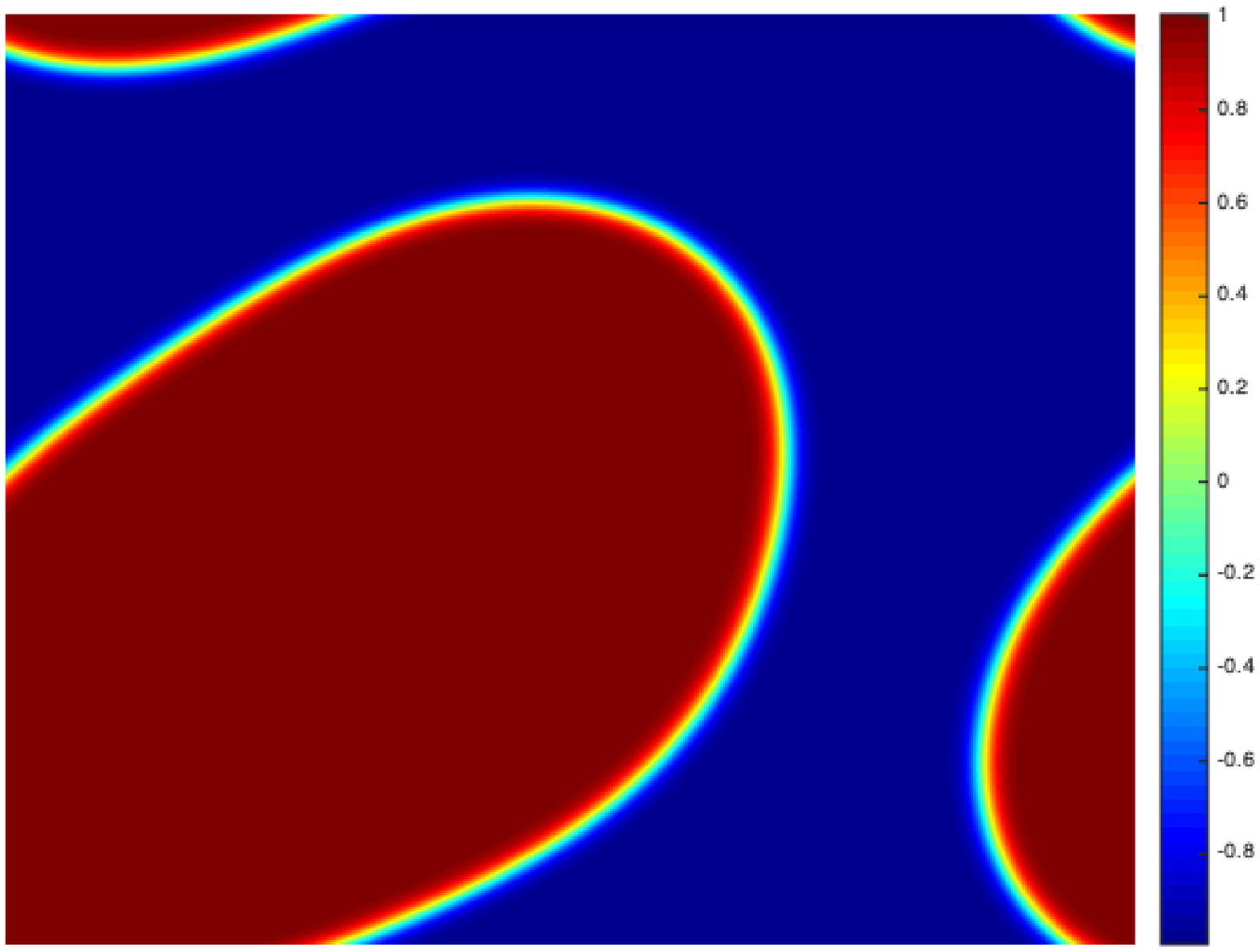}}}\hspace{0.3cm}
\subfigure[$t=300$]{{\includegraphics[width=0.33\textwidth]{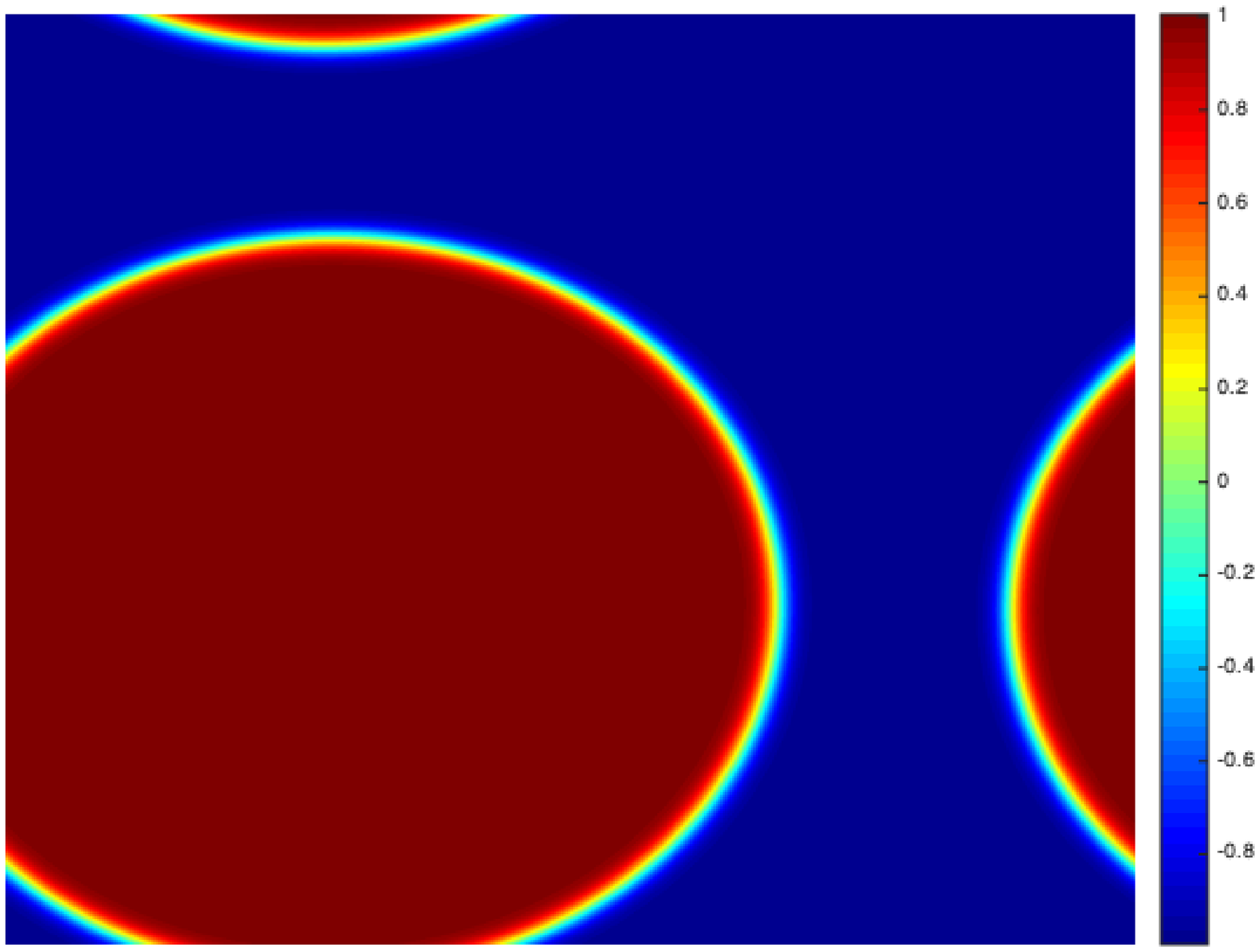}}}}
\caption{The simulated solution at $t=1$, $10$, $50$ and $300$  respectively
for the Cahn-Hilliard equation with an initial  quasi-uniform state in 2D.}
\label{2D1111}
\end{figure*}

\begin{figure*}[!ht]
\centerline{
\subfigure{{\includegraphics[width=0.36\textwidth]{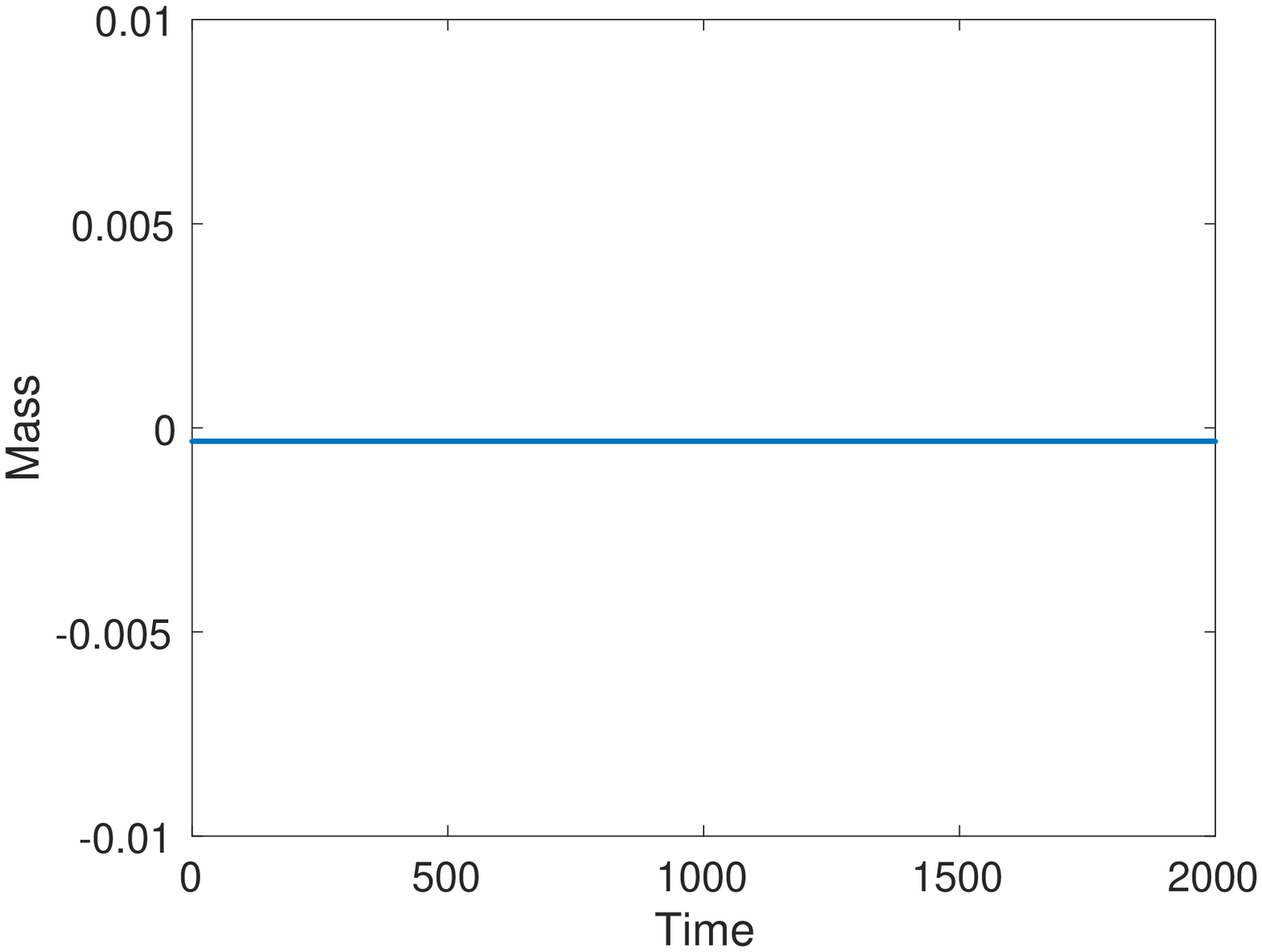}}}\hspace{-0.55cm}
\subfigure{{\includegraphics[width=0.36\textwidth]{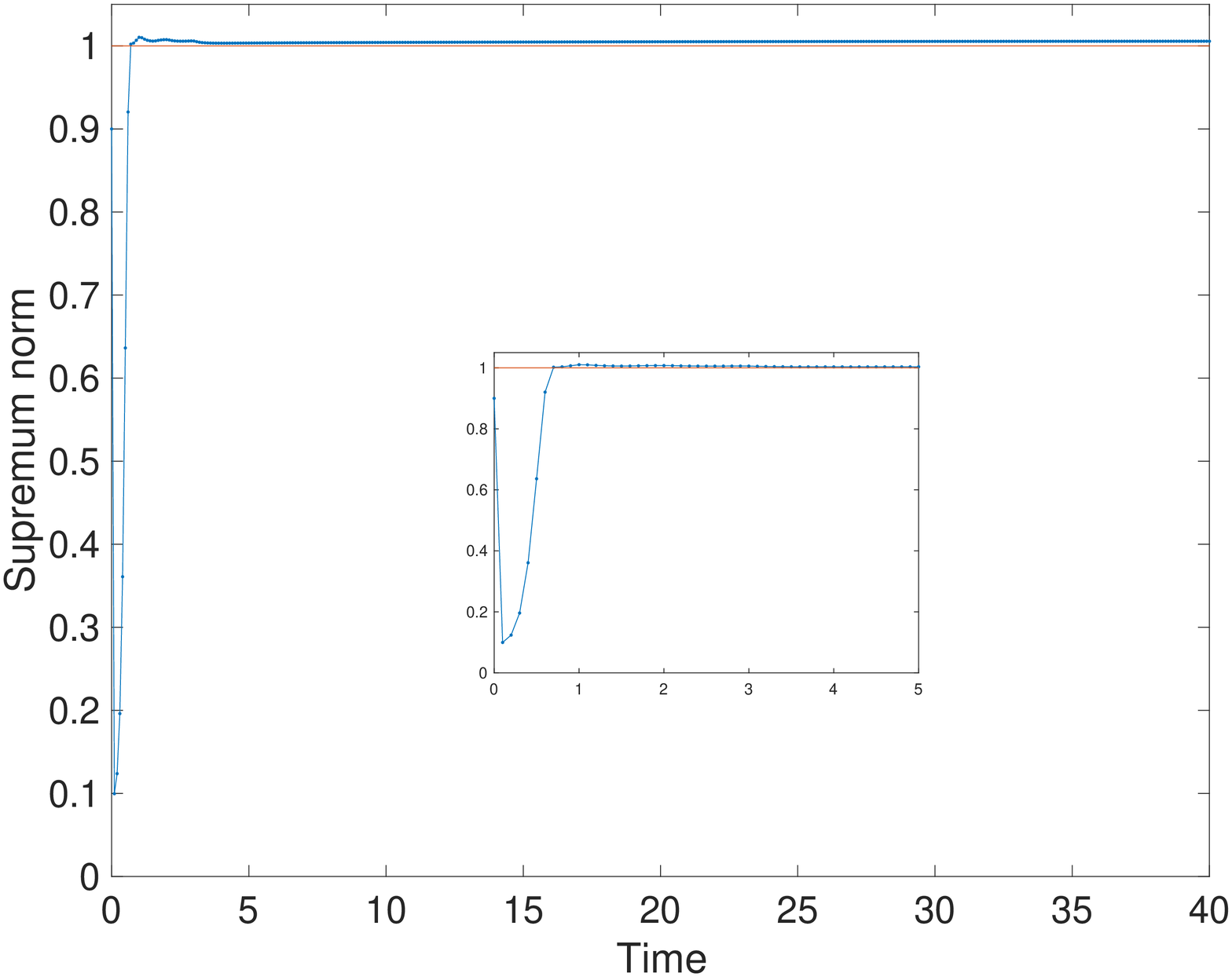}}}\hspace{-0.55cm}
\subfigure{{\includegraphics[width=0.36\textwidth]{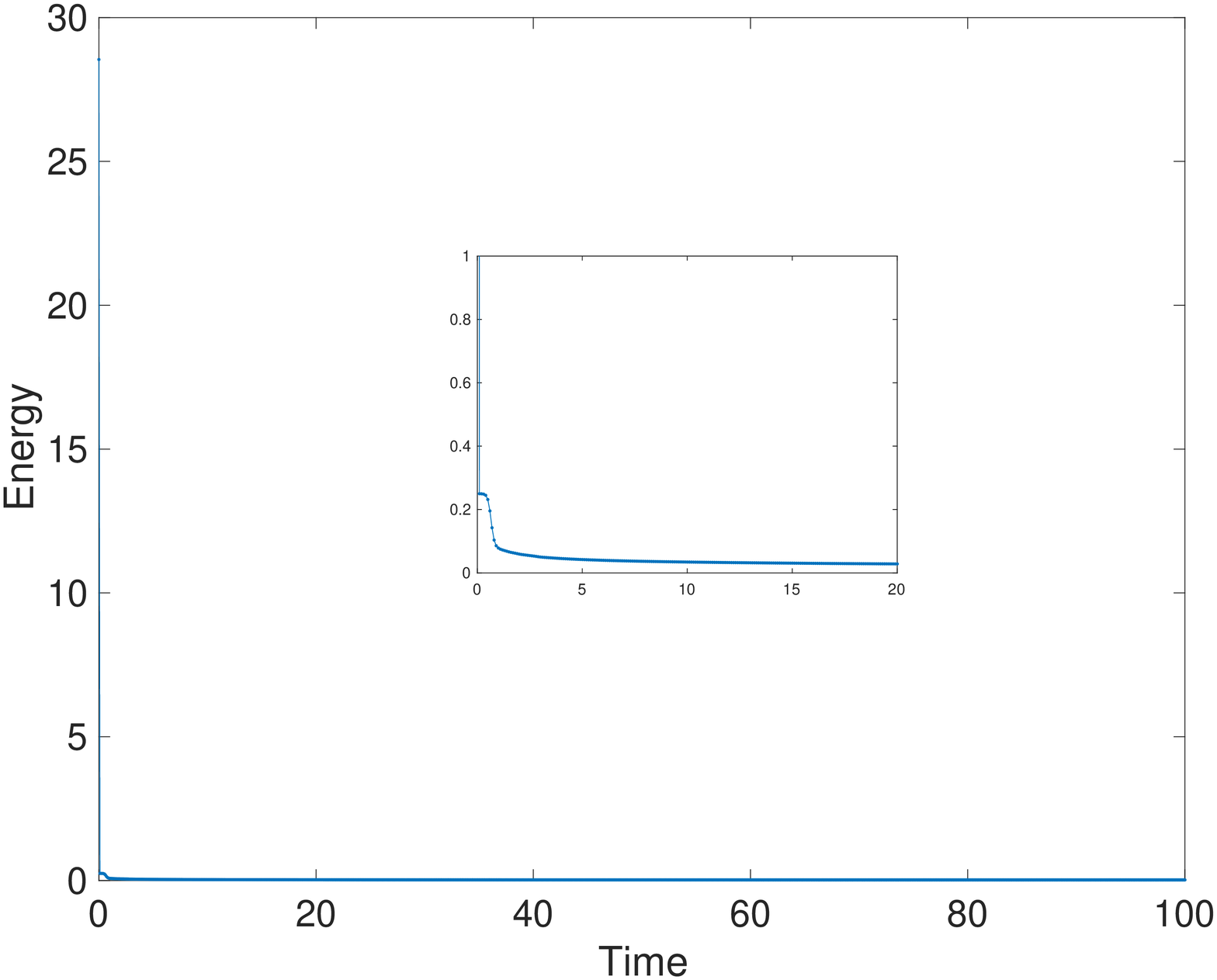}}}}
\caption{Evolutions of the mass, the supremum norm  and  the  energy
of the simulated solutions for the  Cahn-Hilliard equation with an initial  quasi-uniform state in 2D. }\label{CHDR2}
\end{figure*}

Our 3D simulations start with the quasi-uniform initial state $u_0 = 0.9\,\mathrm{rand(\cdot)}$ as well.
\figurename~\ref{3DR1} presents the configurations of the computed solution at $t=1$, $30$, $200$, and $4000$
for the mass-conserving Allen-Cahn equation.
The corresponding evolutions of the mass, the supremum norm and the energy are plotted in \figurename~\ref{3DR2}.
We observed again that the  mass is exactly conserved, the discrete MBP is preserved perfectly, and
the energy decays monotonically along the time.

\begin{figure*}[!ht]
\vspace{-0.35cm}
\centerline{
\subfigure[$t=1$]{{\includegraphics[width=0.38\textwidth]{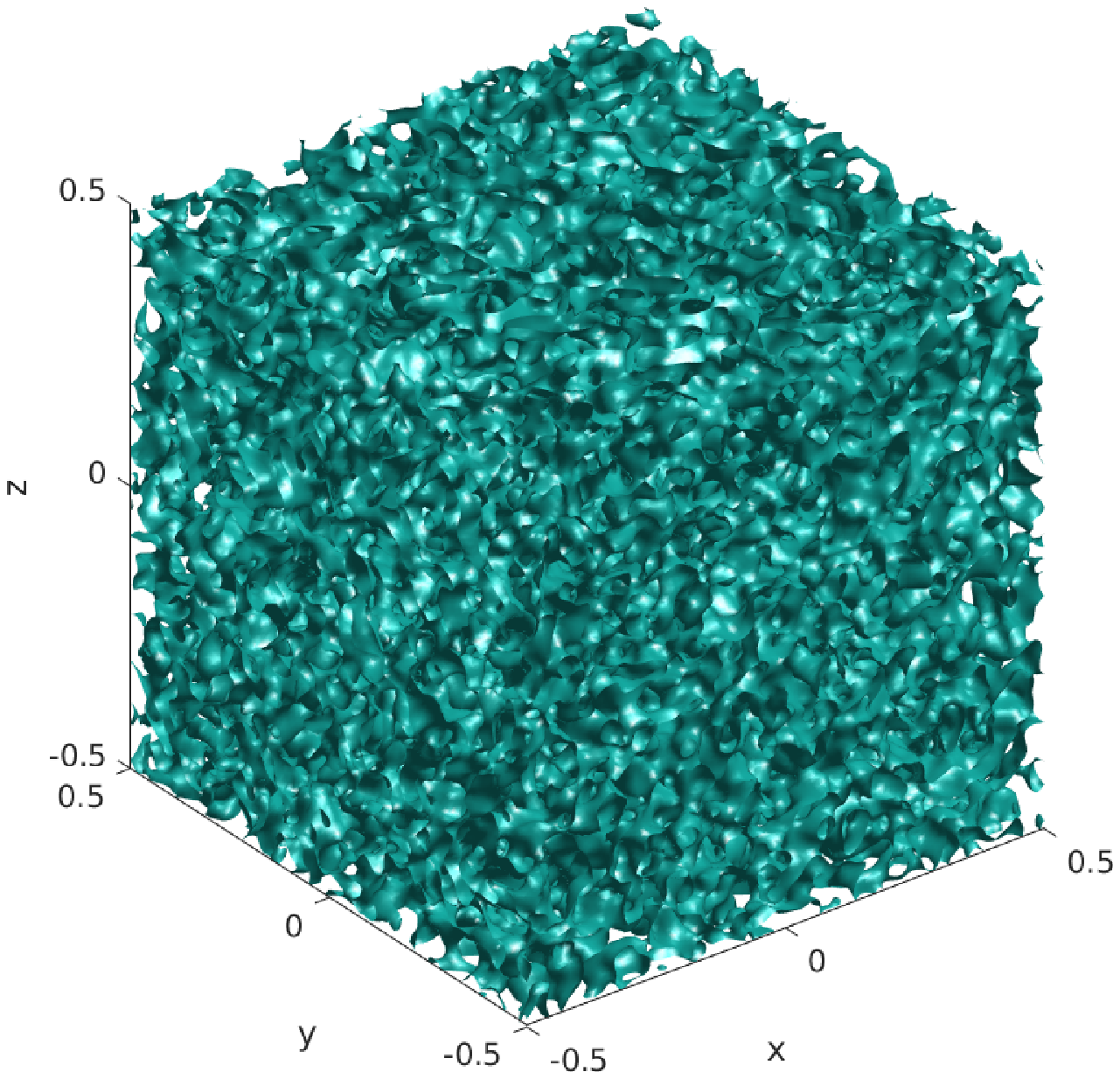}}}
\subfigure[$t=30$]{{\includegraphics[width=0.38\textwidth]{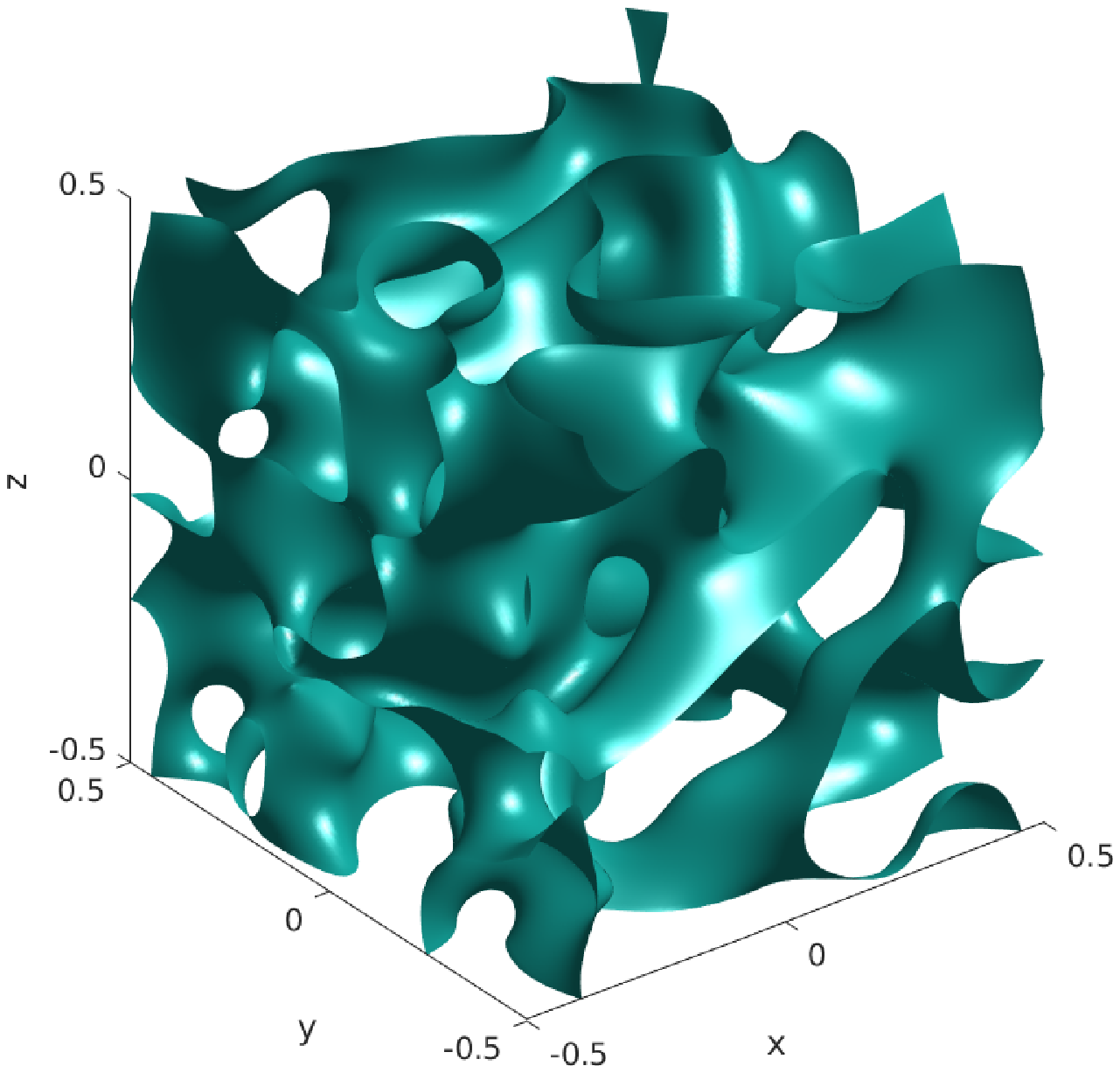}}}}
\vspace{-0.35cm}
\centerline{
\subfigure[$t=200$]{{\includegraphics[width=0.38\textwidth]{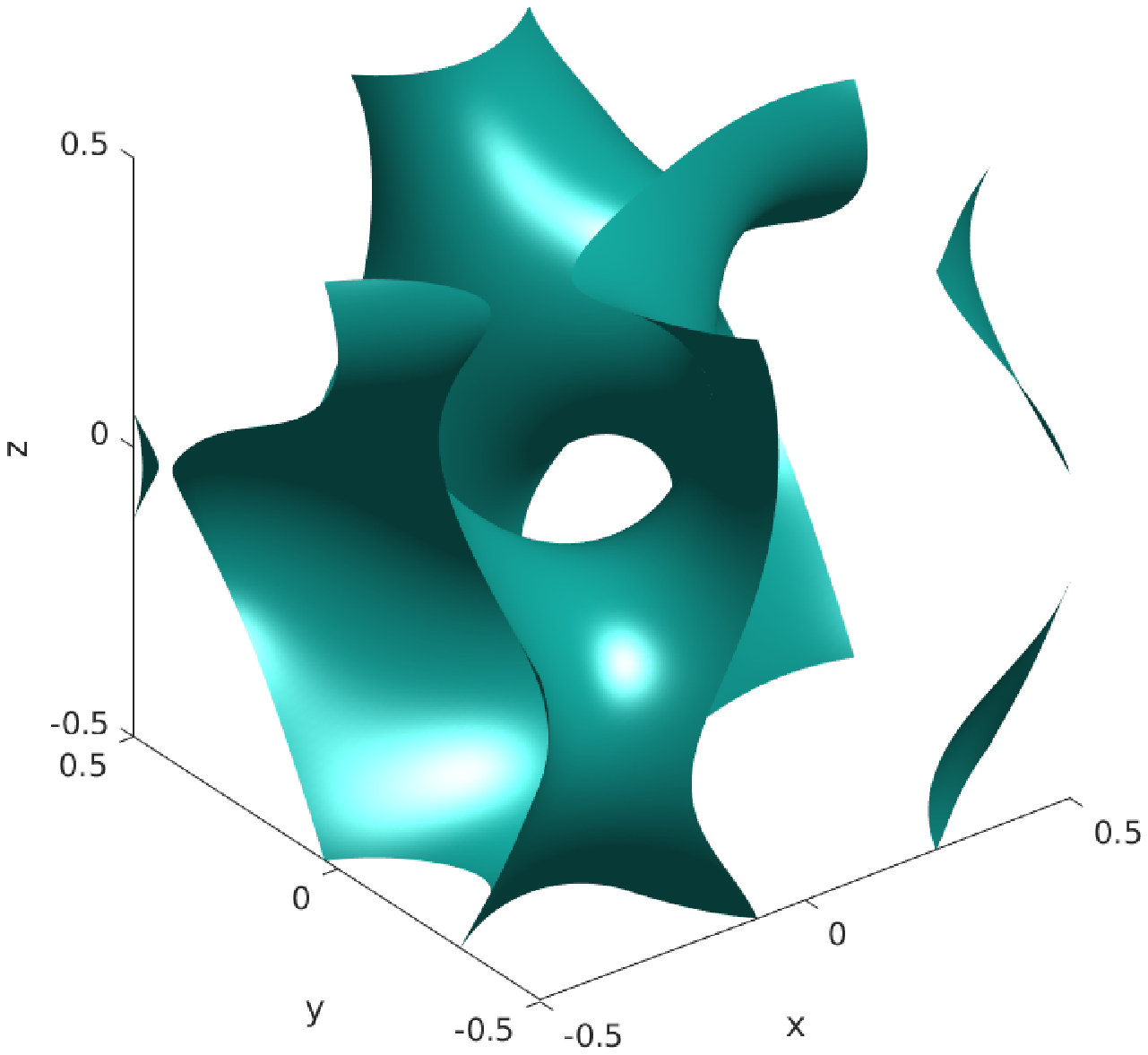}}}
\subfigure[$t=4000$]{{\includegraphics[width=0.38\textwidth]{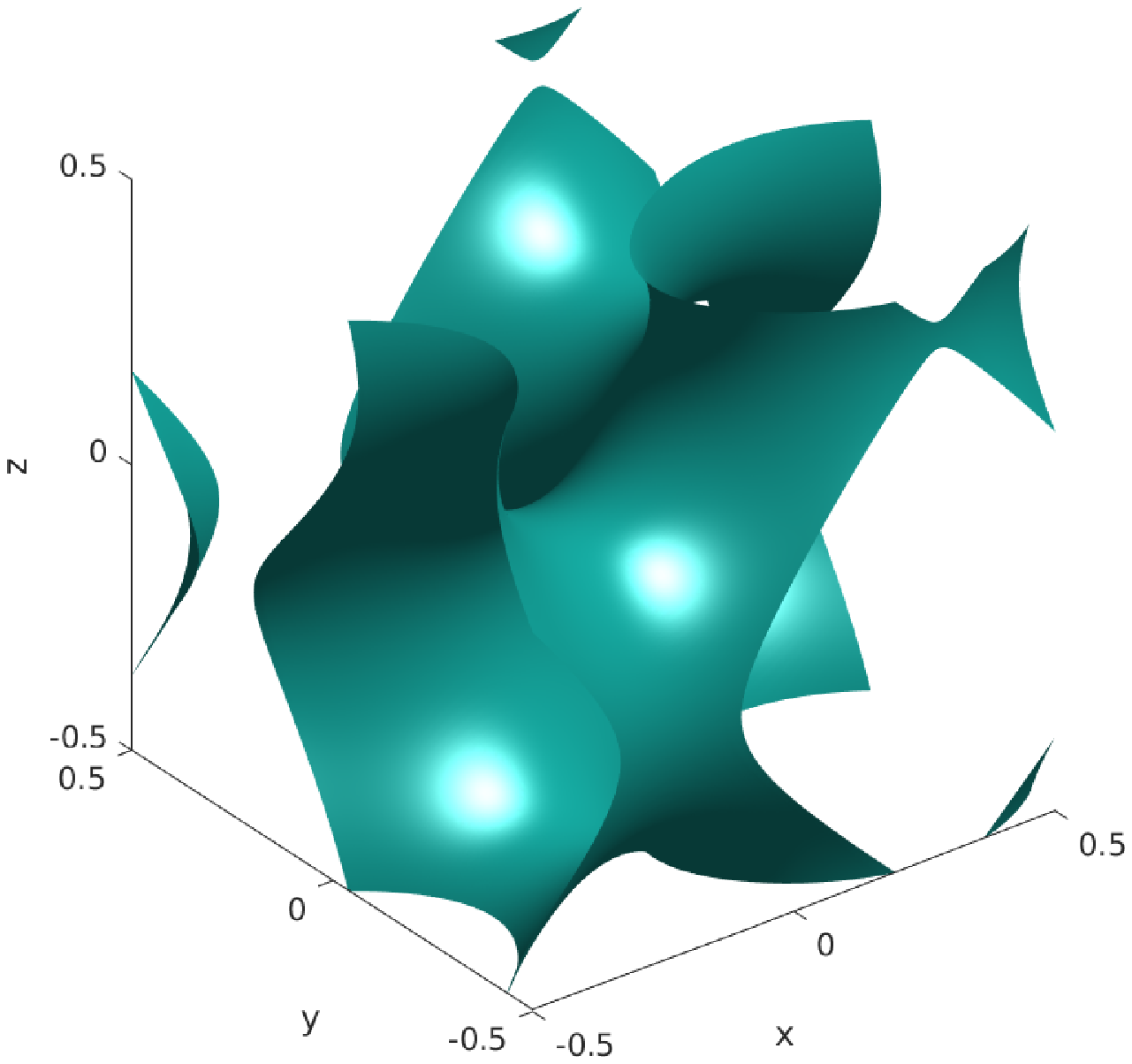}}}}
\caption{The simulated phase structures  at $t=1$, $30$, $200$ and $4000$ respectively
for the mass-conserving Allen-Cahn equation with an initial  quasi-uniform state  in 3D.}\label{3DR1}
\end{figure*}

\begin{figure*}[!ht]
\centerline{
\subfigure{{\includegraphics[width=0.36\textwidth]{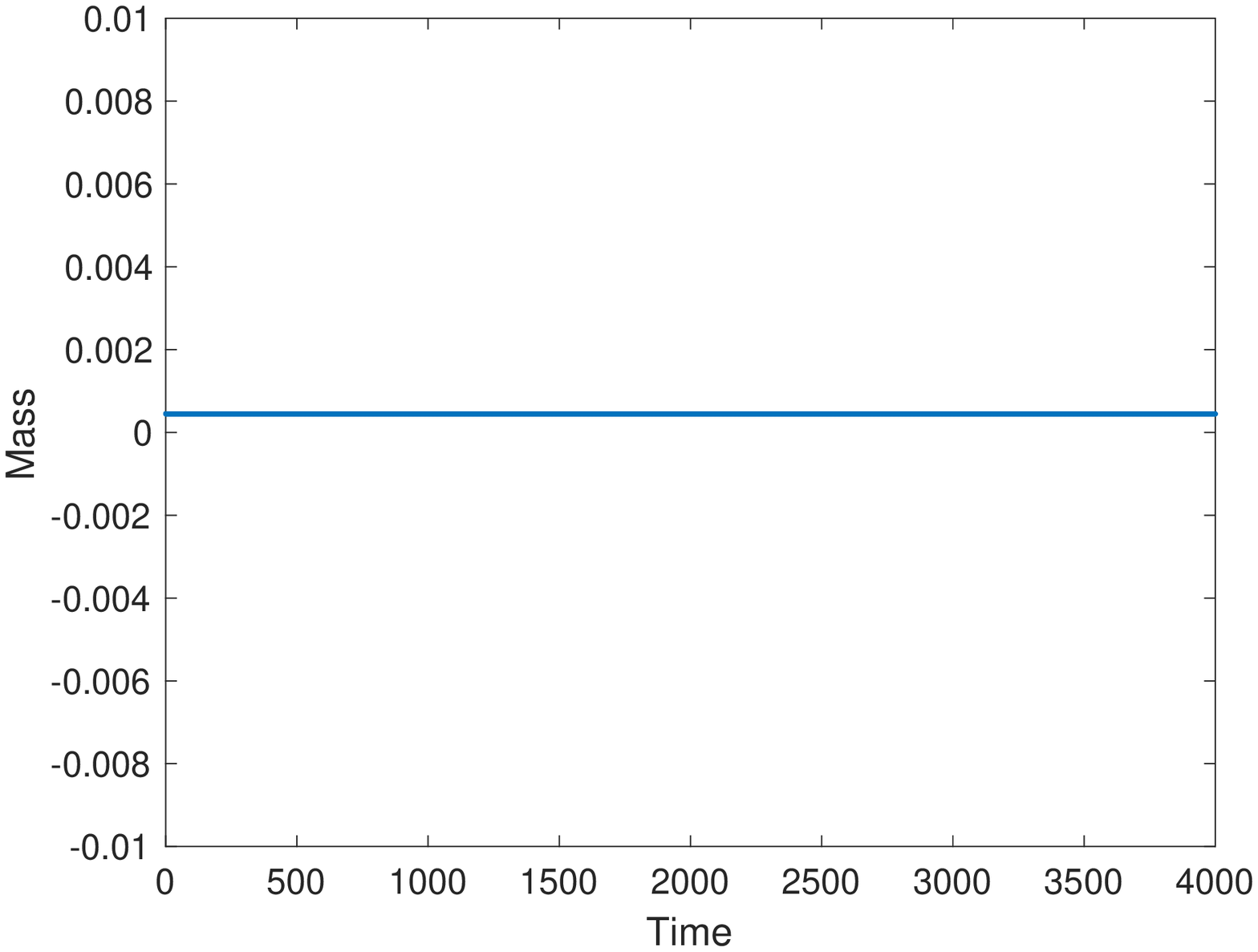}}}\hspace{-0.55cm}
\subfigure{{\includegraphics[width=0.36\textwidth]{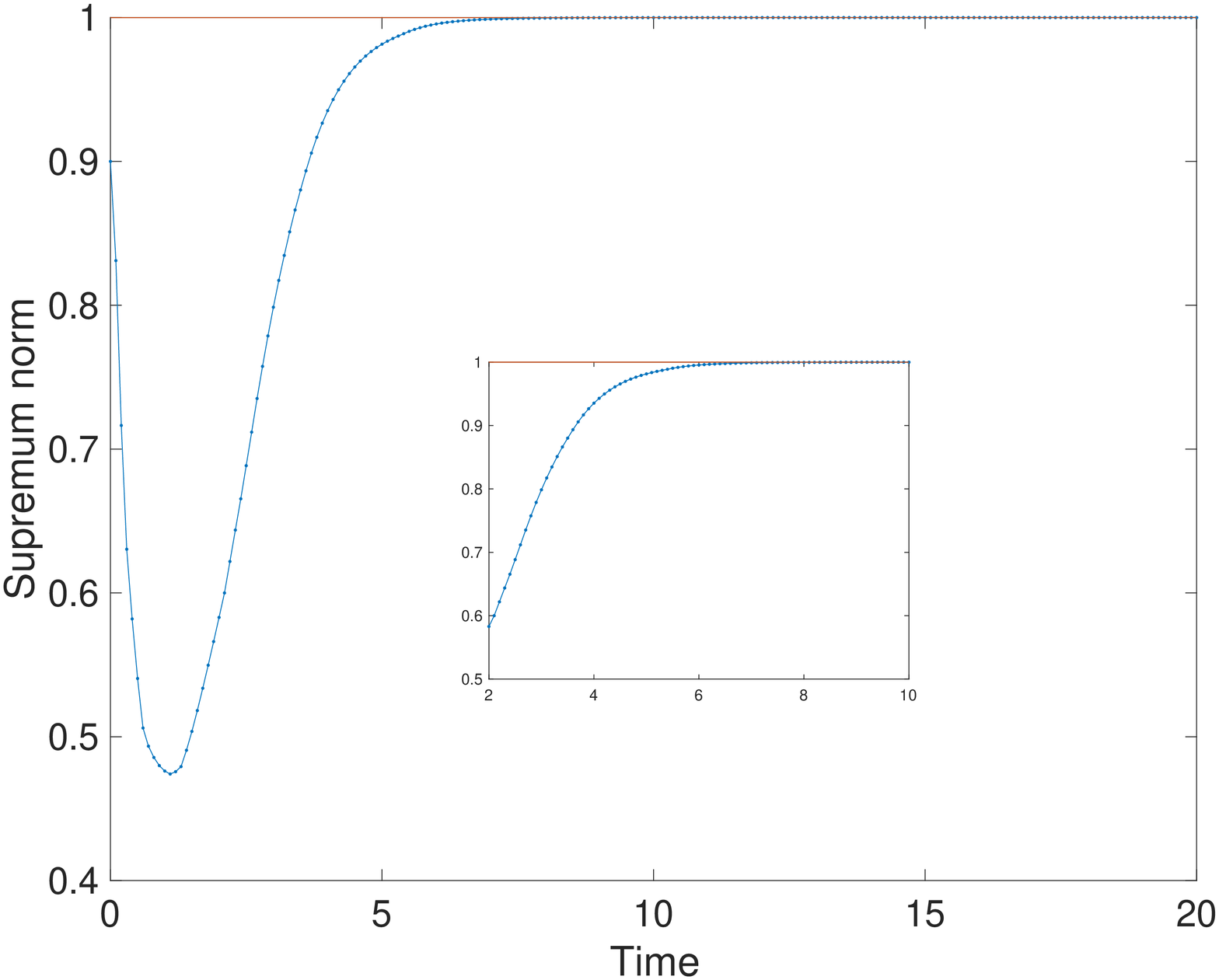}}}\hspace{-0.55cm}
\subfigure{{\includegraphics[width=0.36\textwidth]{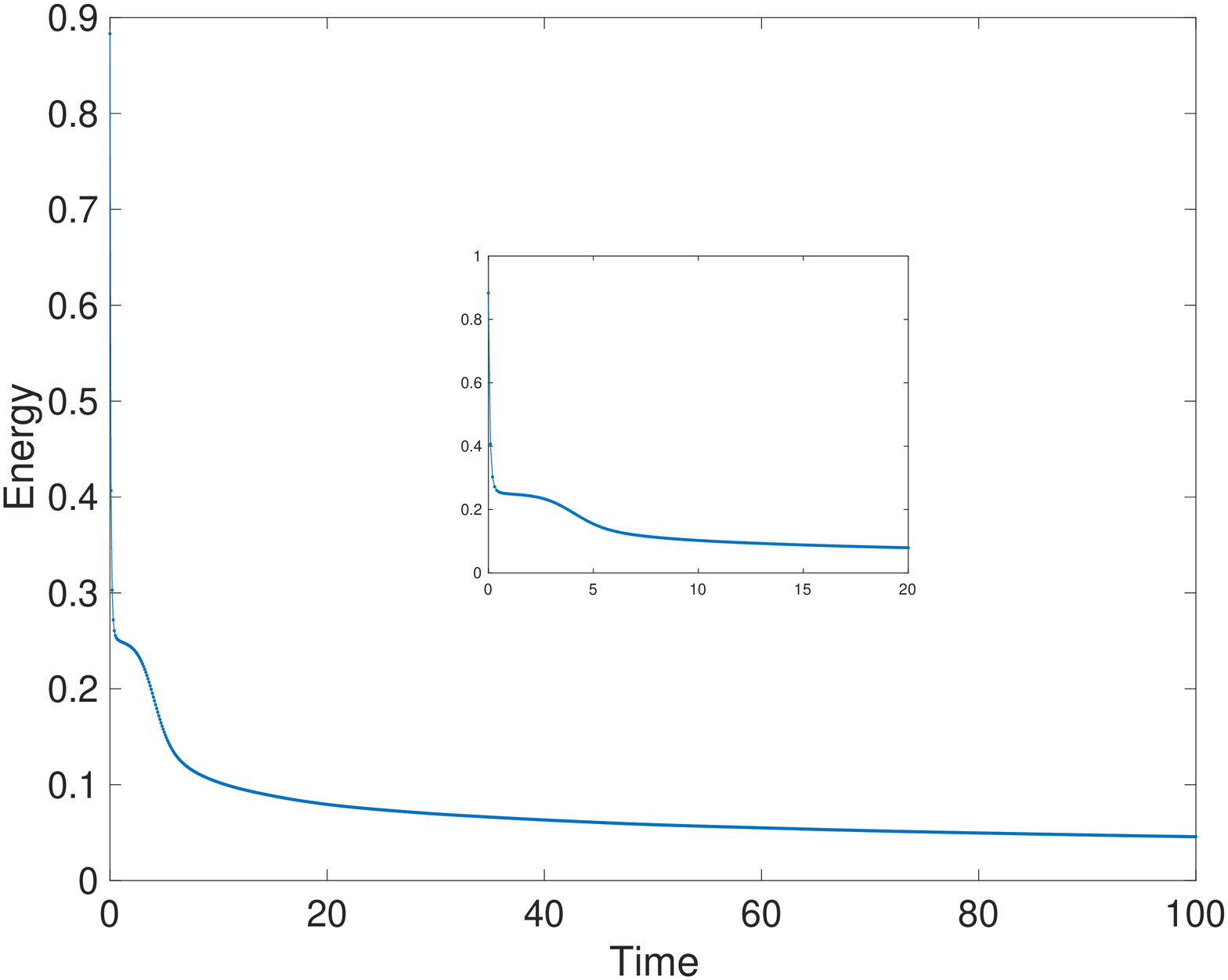}}}}
\caption{Evolutions of the mass, the supremum norm  and  the  energy   of the simulated solutions
for the mass-conserving Allen-Cahn equation with an initial  quasi-uniform state in 3D. }
\label{3DR2}
\end{figure*}

\subsection{The expanding  bubble test}

We use the ETDRK2 scheme to simulate the expansion process of the bubble in 3D, governed by the mass-conserving Allen-Cahn equation beginning with a discontinuous initial configuration
\begin{equation}
u_{0}=\left\{\begin{array}{ll}-0.5, & x^{2}+y^{2}+z^{2}<0.25^{2}, \\ 0.5, & \text { otherwise. }\end{array}\right.
\end{equation}
The temporal and spatial step size are set as $\tau= 0.1$ and $h = 1/256$.
\figurename~\ref{3DB1} presents the simulated process of the expanding bubble,
that is, the isosurface views of the numerical solutions at $t=1$, $10$, $15$, and $100$, respectively.
\figurename~\ref{3DB3} illustrates the evolutions of  the mass, the supremum norm, the energy and the radius of the bubble along the time.
It is again observed that the mass is conserved,  the discrete MBP is well preserved, and the energy decays monotonically.
The radius of the bubble increases monotonically and the steady state is reached (a bubble with  radius $r\approx 0.407$  as expected\cite{LJCF20}) after about $t=18$.

\begin{figure*}[!ht]
\vspace{-0.35cm}
\centerline{
\subfigure[$t=1$]{{\includegraphics[width=0.38\textwidth]{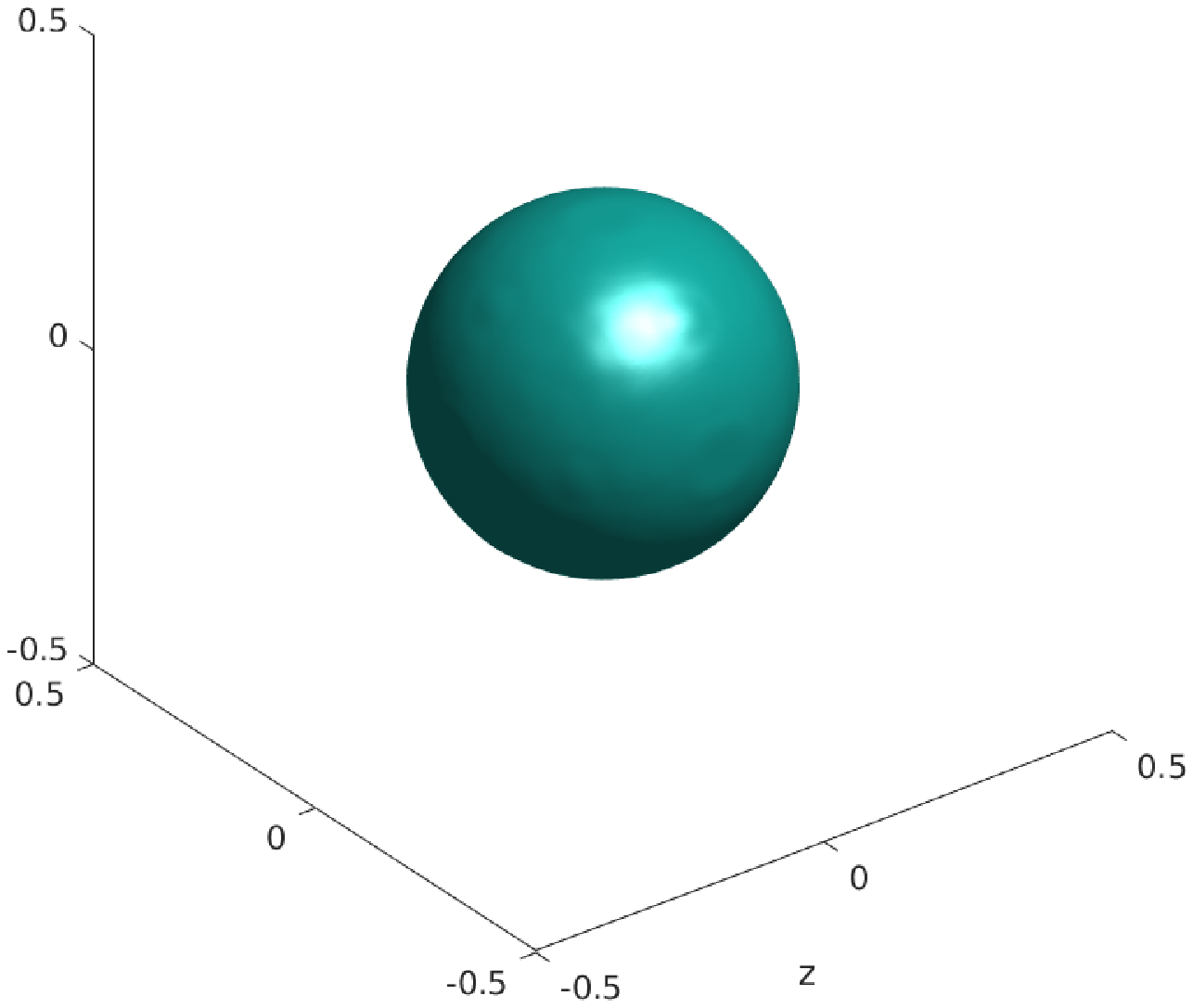}}}
\subfigure[$t=10$]{{\includegraphics[width=0.38\textwidth]{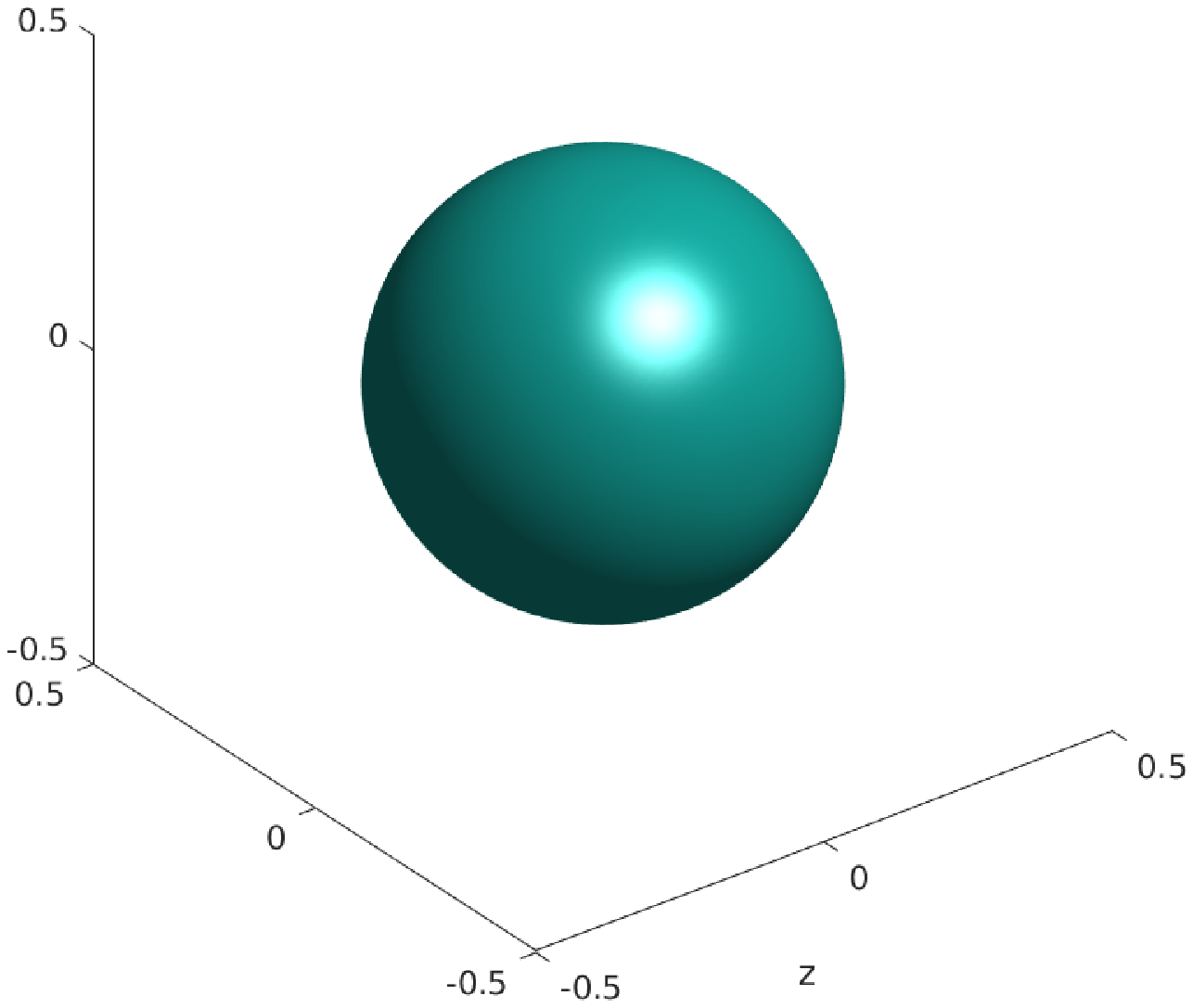}}}}
\vspace{-0.35cm}
\centerline{
\subfigure[$t=15$]{{\includegraphics[width=0.38\textwidth]{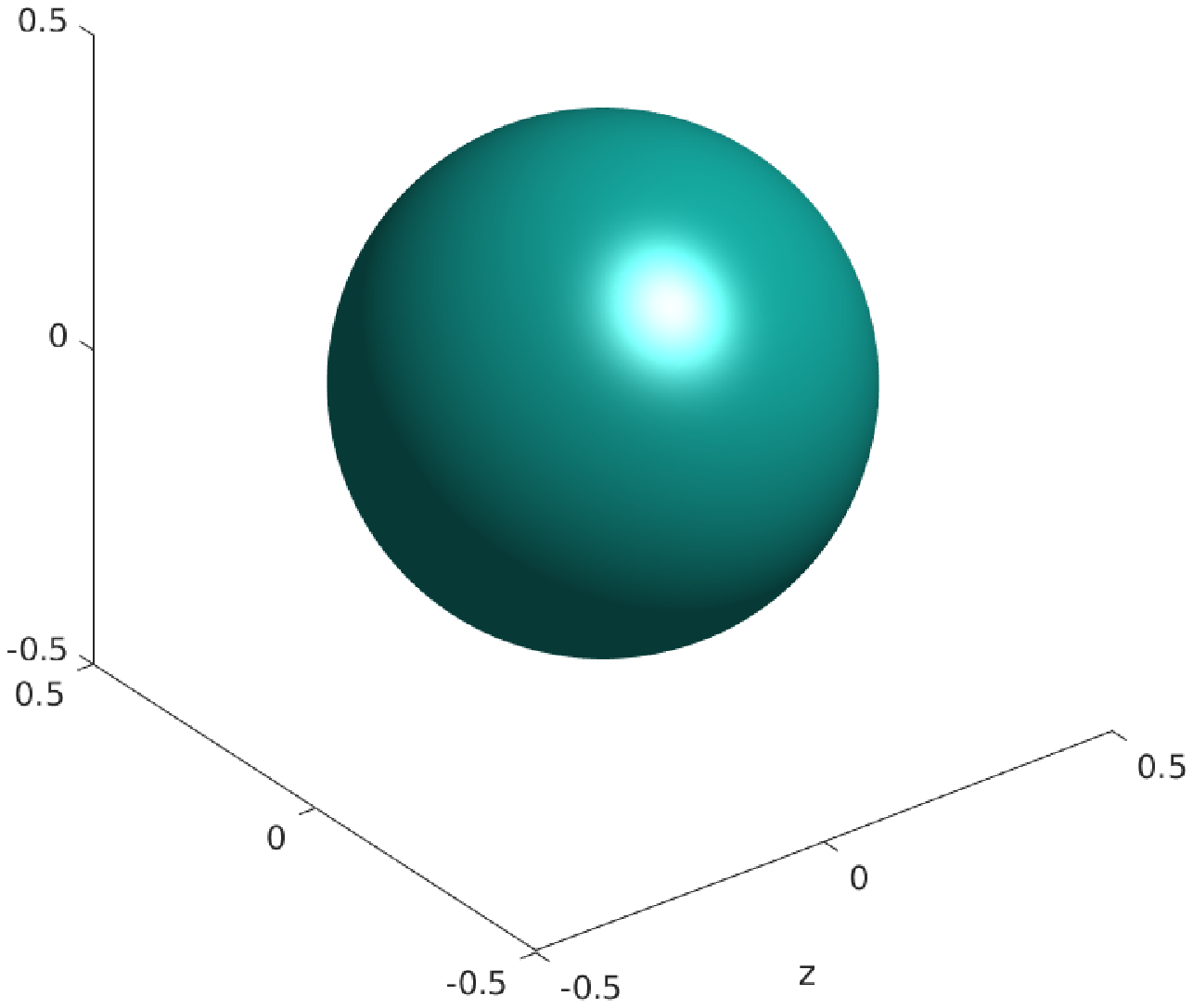}}}
\subfigure[$t=100$]{{\includegraphics[width=0.38\textwidth]{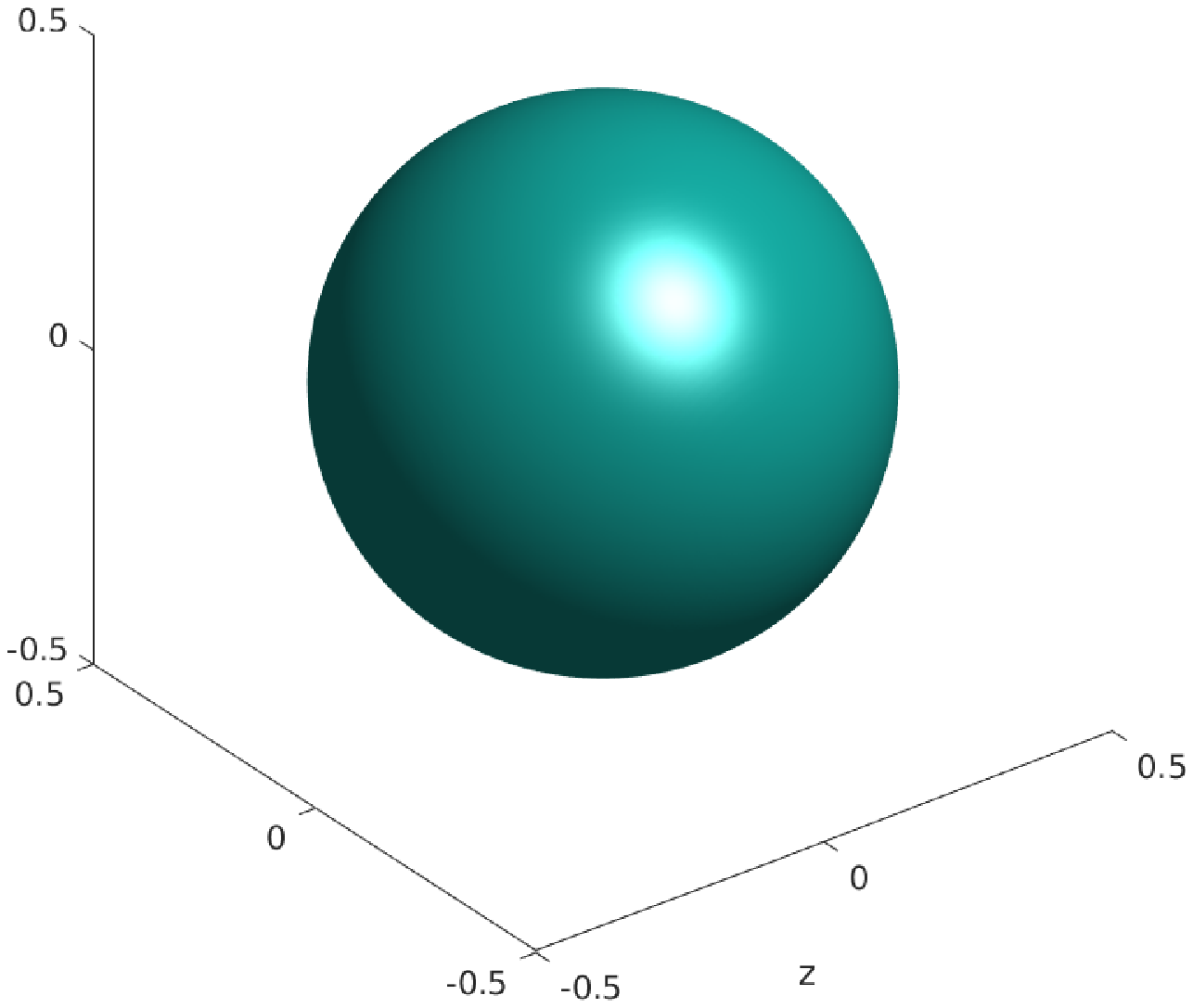}}}}
\caption{The simulated phase structures at $t=1$, $10$, $15$ and $100$ respectively
for the expanding bubble test in 3D.}\label{3DB1}
\end{figure*}

\begin{figure*}[!ht]
\centerline{
\subfigure{{\includegraphics[width=0.37\textwidth]{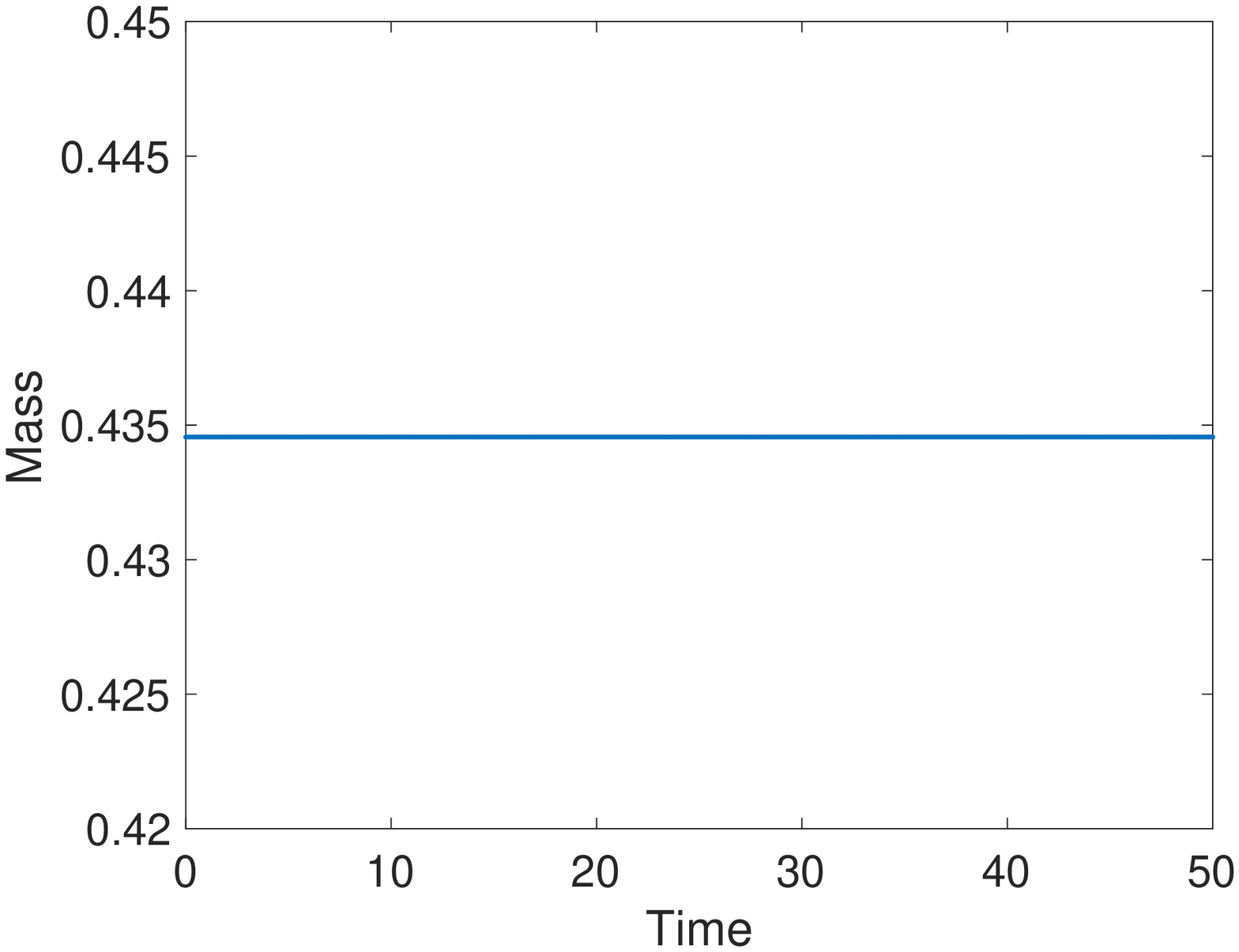}}}\hspace{-0.25cm}
\subfigure{{\includegraphics[width=0.37\textwidth]{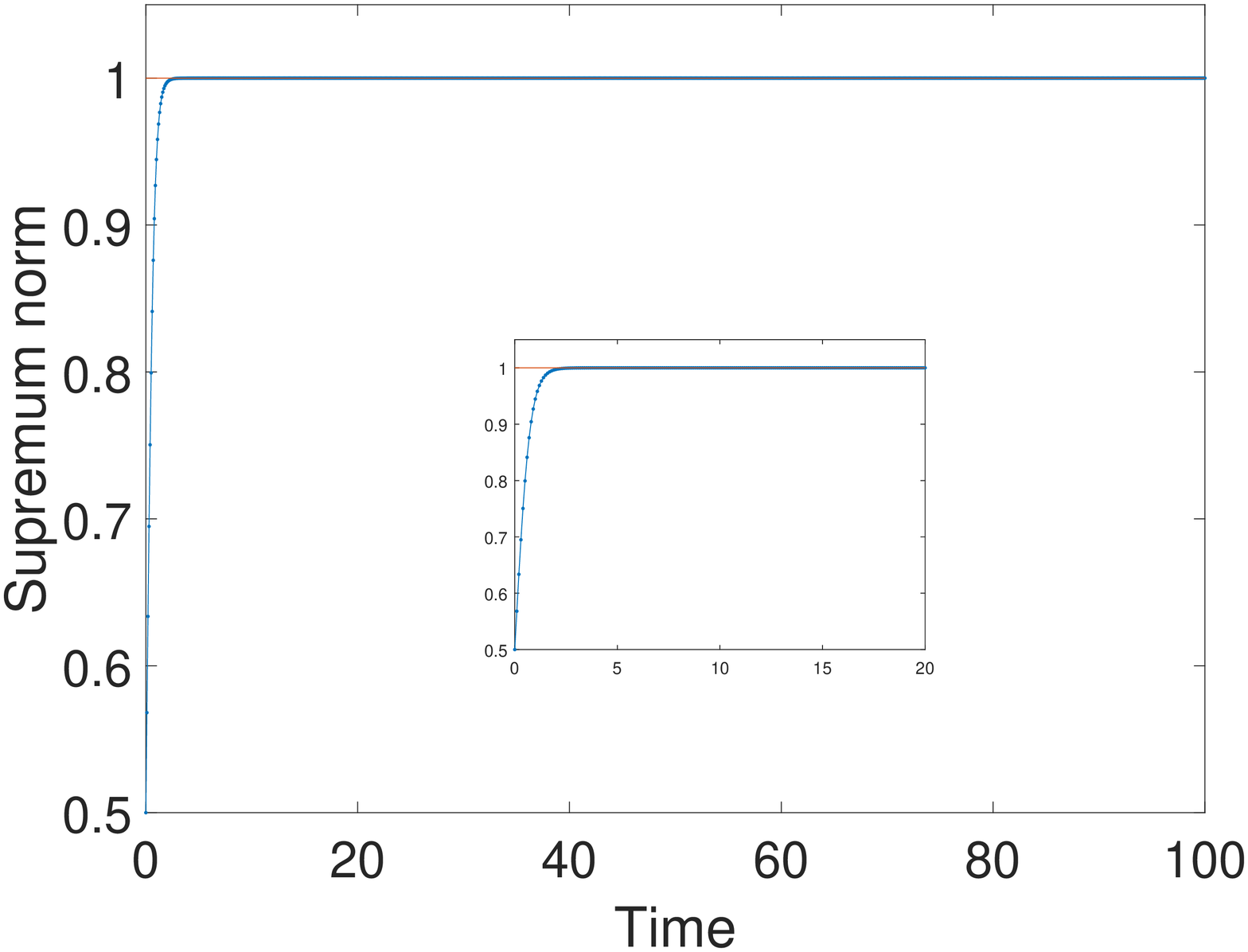}}}}
\centerline{
\subfigure{{\includegraphics[width=0.37\textwidth]{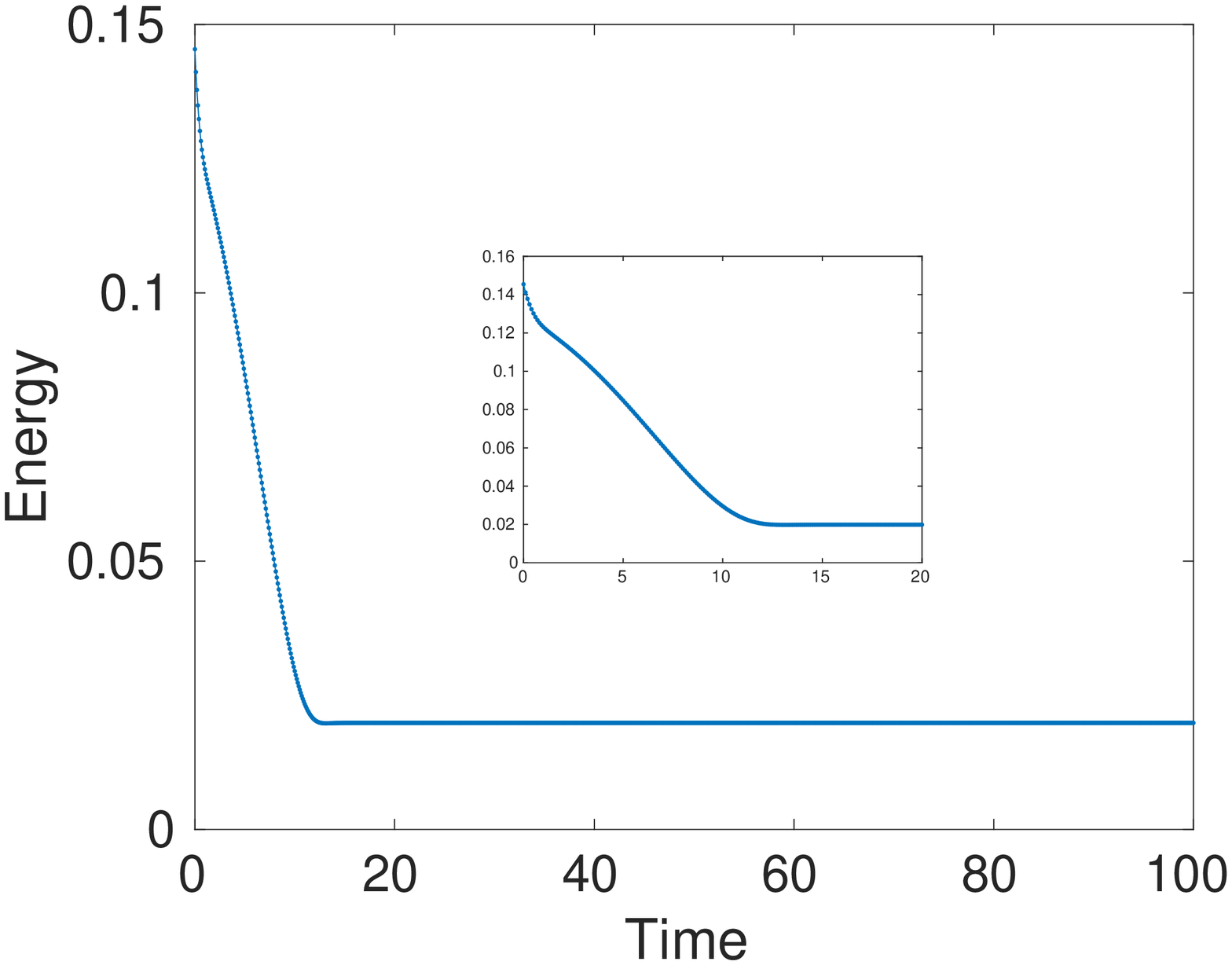}}}\hspace{-0.25cm}
\subfigure{{\includegraphics[width=0.37\textwidth]{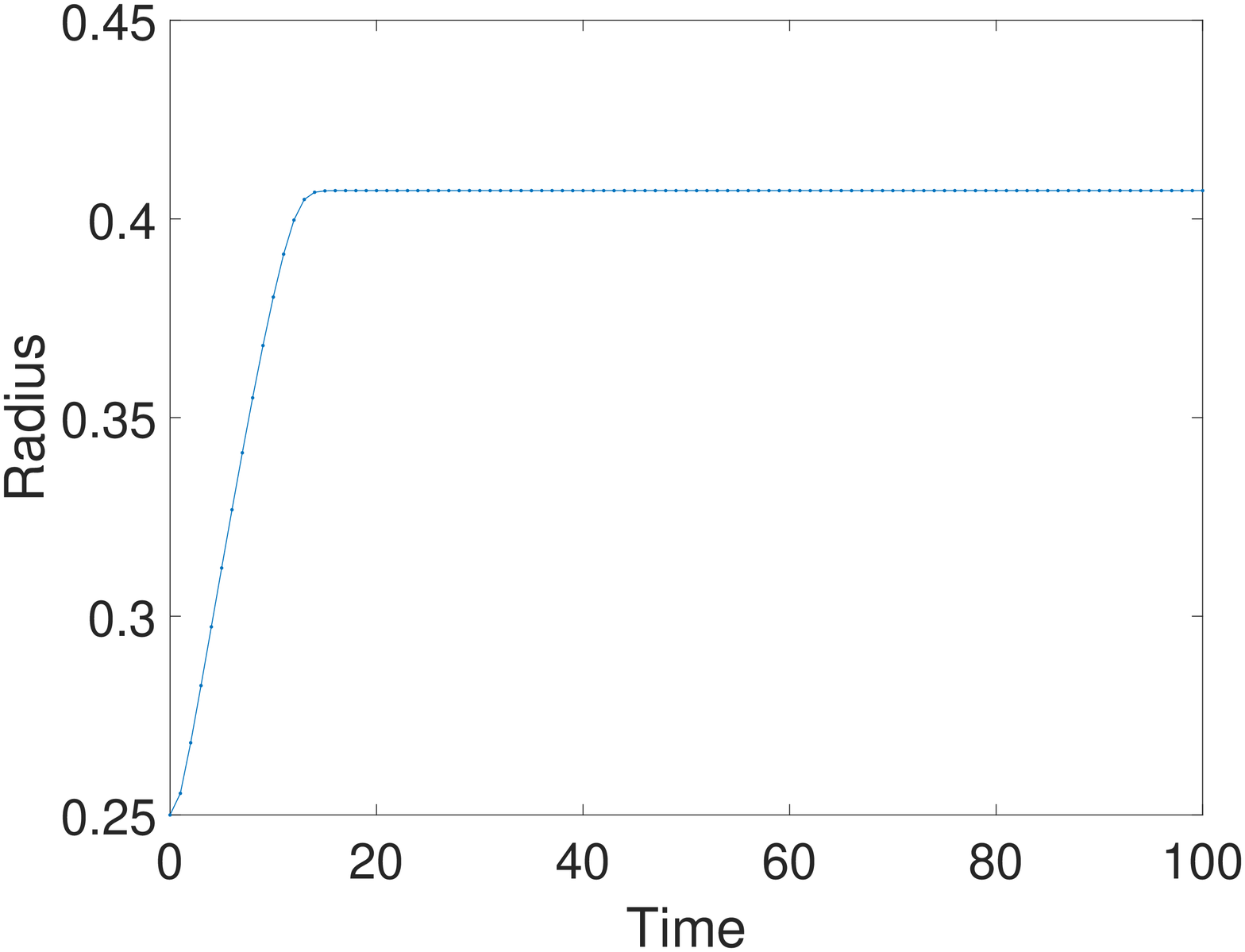}}}}
\caption{Evolutions of the mass,  the supremum norm, the energy  and the radius
of the simulated solutions for  the expanding bubble test in 3D. }
\label{3DB3}
\end{figure*}

\section{Conclusions}\label{sec5}

In this paper, we propose and analyze  first- and second- order linear schemes for solving the  mass-conserving Allen-Cahn equation with local and nonlocal effects (in the double-well potential case), which are based on the combination of the linear stabilizing technique and the exponential time differencing method. We prove that the proposed schemes are  unconditionally MBP-preserving and mass-conserved in the time-discrete sense. Error estimates of these schemes are also rigorously derived under some assumptions. It remains an open problem whether a more delicate analysis can relieve the extra assumption on the numerical solutions $\{u^n\}$ in Theorems \ref{thm-etd1-conv} and \ref{thm-etd2rk-conv} as discussed in Remark \ref{addcond}. In addition, it is worth mentioning that the Flory-Huggions potential is also  widely-used in the classic Allen-Cahn model and how to extend the current work to that case is  subject to future investigation as well.

\section*{Acknowledgments}

L. Ju's work is partially supported by U.S. National Science Foundation  under grant numbers DMS-1818438 and DMS-2109633.
J. Li's work is partially supported by National Natural Science Foundation of China under grant number 61962056.
X. Li's work is partially supported by National Natural Science Foundation of China under grant number 11801024.


\end{document}